\RequirePackage{fix-cm}
\documentclass[11pt]{article}
\usepackage[top=3cm, bottom=3cm, left=2cm, right=2cm]{geometry}

\usepackage{graphicx}
\usepackage{enumitem}
\usepackage{amsmath,amssymb}
\usepackage{multirow}
\newcommand{\tabincell}[2]{\begin{tabular}{@{}#1@{}}#2\end{tabular}}
\usepackage{algorithmic}
\usepackage{algorithm,float}
\usepackage{subfigure}
\usepackage{pdflscape}
\usepackage{threeparttable}
\usepackage{siunitx}
\usepackage{tikz}
\usepackage{hyperref}
\usepackage{caption}

\usepackage{amsthm}
\newtheorem{proposition}{Proposition}

\newtheorem{assumption}{Assumption}
\newtheorem{theorem}{Theorem}

\newcommand*\showfontsize{\fbox{\the\fontdimen6\the\font}}



\usepackage{caption}
\begin{document}

\title{Adaptive sieving: A dimension reduction technique for sparse optimization problems}


\author{
	Yancheng Yuan\thanks{Department of Applied Mathematics, The Hong Kong Polytechnic University, Hung Hom, Hong Kong. \texttt{yancheng.yuan@polyu.edu.hk}} 
	\and
	Meixia Lin\thanks{(Corresponding author) Engineering Systems and Design, Singapore University of Technology and Design, Singapore. \texttt{meixia\_lin@sutd.edu.sg}} 
	\and
	Defeng Sun\thanks{Department of Applied Mathematics, The Hong Kong Polytechnic University, Hung Hom, Hong Kong. \texttt{defeng.sun@polyu.edu.hk}} 
	\and
	Kim-Chuan Toh\thanks{Department of Mathematics and Institute of Operations Research and Analytics, National University of Singapore, Singapore. \texttt{mattohkc@nus.edu.sg}}
}


\maketitle

\begin{abstract}
In this paper, we propose an adaptive sieving (AS) strategy for solving general sparse machine learning models by effectively exploring the intrinsic sparsity of the solutions, wherein only a sequence of reduced problems with much smaller sizes need to be solved. We further apply the proposed AS strategy to generate solution paths for large-scale sparse optimization problems efficiently. We establish the theoretical guarantees for the proposed AS strategy including its finite termination property. Extensive numerical experiments are presented in this paper to demonstrate the effectiveness and flexibility of the AS strategy to solve large-scale machine learning models.
\end{abstract}


\vspace{0.3em}
\noindent\textbf{MSC Classes:} 90C06, 90C25, 90C90

\section{Introduction}
\label{sec:intro}
Consider the convex composite optimization problems of the following form:
\begin{align}
	\label{eq: lasso_model}
	\min_{x\in \mathbb{R}^n}\ \displaystyle \Big\{\Phi(x) + P(x)\Big\},
\end{align}
where $\Phi:\mathbb{R}^n\rightarrow \mathbb{R}$ is a convex twice continuously differentiable function and $P: \mathbb{R}^n \to (-\infty,+\infty]$ is a closed and proper convex function. The optimization problems in this form cover a wide class of models in modern data science applications and statistical learning. In practice, the regularizer $P(\cdot)$ is usually chosen to enforce sparsity with desirable structure in the estimators obtained from the model, especially in high-dimensional cases. For example, the Lasso regularizer \cite{tibshirani1996regression} is proposed to force element-wise sparsity in the predictors, and the group lasso regularizer \cite{yuan2006model} is proposed to impose group-wise sparsity. Moreover, many more complicated regularizers have also been proposed to study other structured sparsity, such as the sparse group lasso regularizer \cite{eldar2009robust,jacob2009group}, the Sorted L-One Penalized Estimation (SLOPE) \cite{bogdan2015slope} and the exclusive lasso regularizer \cite{zhou2010exclusive,kowalski2009sparse}.

Currently, many popular first-order methods have been proposed to solve the problems in the form of \eqref{eq: lasso_model}, such as the accelerated proximal gradient (APG) method \cite{beck2009fast}, the alternating direction method of multipliers (ADMM) \cite{eckstein1992douglas,glowinski1975approximation} and the block coordinate descent (BCD) method \cite{tseng1993dual,grippo2000convergence,sardy2000block}. These algorithms are especially popular in machine learning and statistics in recent years. However, first-order methods are often not robust and generally are only able to deliver low accuracy solutions. Other than first-order methods, some second-order methods have been developed for solving convex composite optimization problems in the form of \eqref{eq: lasso_model}, including the coderivative-based generalized Newton method \cite{khanh2024globally}, the forward-backward quasi-Newton method \cite{stella2017forward}, the proximal trust-region method \cite{baraldi2023proximal}, and second-order algorithms based on the augmented Lagrangian method \cite{li2018highly,cui2018composite,tang2023proximal,liu2024inexact}.

However, both the first-order and second-order algorithms for solving the problems in the form of \eqref{eq: lasso_model} will face significant challenges when scaling to high-dimensional problems. Fortunately, the desired solutions to the model in practice are usually highly sparse, where the nonzero entries correspond to selected features. Inspired by this property, in this paper, we propose an adaptive sieving strategy for solving sparse optimization models in the form of \eqref{eq: lasso_model}, by sieving out a large proportion of inactive features to significantly reduce the dimension of the problems, which can then highly accelerate the computation. The proposed AS strategy does not depend on the specific form of the regularizer $P(\cdot)$, as long as the proximal mapping of $P(\cdot)$ can be computed in an efficient way. It is worth emphasizing that our AS strategy can be applied to any solver providing that it is able to solve the reduced problems to the required accuracy. Numerical experiments demonstrate that our AS strategy is very effective in reducing the problem dimensions and accelerating the computation for solving sparse optimization models.

We will then apply the proposed AS strategy to generate solution paths for large-scale machine learning models, which is necessary for model selection via cross-validation in practice. In particular, we are interested in solving the following problem for a sequence of values of the parameter $\lambda$:
\begin{align}
	\label{eq: lasso_model2}
	\min_{x\in \mathbb{R}^n}\ \displaystyle \Big\{\Phi(x) + \lambda P(x)\Big\}.\tag{$P_{\lambda}$}
\end{align}
Here the hyper-parameter $\lambda > 0$ is added to control the trade-off between the loss and the sparsity level of the solutions. To reduce the computation time of obtaining a solution path, especially for high-dimensional cases, various feature screening rules, which attempt to drop some inactive features based on prior analysis, have been proposed. Tibshirani et al. \cite{tibshirani2012strong} proposed a strong screening rule (SSR) based on the ``unit-slope" bound assumption for the Lasso model and the regression problem with the elastic net regularizer. This idea has been extended to the SLOPE model recently \cite{SSR-SLOPE}. Compared to this unsafe screening rule where some active features may be screened out by mistake, safe screening rules have been extensively studied. The first safe screening rule was proposed by El Ghaoui et al. \cite{ghaoui2010safe} for Lasso models. Later on, Wang et al. \cite{wang2013lasso,wang2015lasso} proposed a dual polytope projection based screening rule (DPP) and an enhanced version (EDPP) for Lasso and group lasso models via carefully analyzing the geometry of the corresponding dual problems. Other safe screening rules, like Sphere test \cite{xiang2016screening}, have been proposed via different strategies for estimating compact regions containing the optimal solutions to the dual problems. The safe screening rules will not exclude active features, but in many instances they only exclude a small subset of inactive features due to their conservative nature. In addition, the safe screening rules are usually not applicable to generate a solution path for a sequence of the hyperparameters with large gaps. Recently, Zeng et al. \cite{zeng2021hybrid} combined the SSR and the EDPP to propose a hybrid safe-strong screening rule, which was implemented in an R package {\tt biglasso} \cite{zeng2017biglasso}. However, these screening rules are usually problem specific and they are difficult to be applied to the models with general regularizers, like the exclusive lasso regularizer, as they are highly dependent on the separability and positive homogeneity of the regularizers. Also, they implicitly require the reduced problems to be solved exactly.

Based on our proposed AS strategy for solving the problem \eqref{eq: lasso_model}, we design a path generation method for generating solution paths of the machine learning model \eqref{eq: lasso_model2}, wherein a sequence of reduced problems with much smaller sizes need to be solved. Our path generation method exhibits great improvement over the existing screening rules in three aspects. First, it applies to the models with general regularizers, including those are non-separable or not positively homogeneous. Second, it does not require the reduced problems to be solved exactly. Third, as we will see in the numerical experiments, it is more aggressive, and is able to sieve out more inactive features. Here, it is worth mentioning that, the aforementioned screening rules are only for generating the solution path and they cannot be applied directly for solving a single problem with a fixed parameter.

The remaining part of the paper is organized as follows. In Section \ref{sec: screening}, we propose the adaptive sieving strategy for solving the sparse optimization models in the form of \eqref{eq: lasso_model}, followed by its theoretical analysis including the finite termination property in Section \ref{sec: analysis}. In Section \ref{sec: inexact}, we will design a path generation method for general sparse optimization models based on the proposed AS strategy. Section \ref{sec: experiments} provides the numerical performance of the proposed AS strategy for solving various machine learning models on both synthetic data and real data. Experiments are also conducted to demonstrate the superior performance of the path generation method based on the AS strategy compared to other popular screening rules. Finally, we conclude the paper.

\vspace{0.2cm}
\noindent\textbf{Notations:} Denote the set $[n] = \{1,2,\cdots,n\}$. For any $I\subseteq [n]$, $\overline{I}$ denotes the complement of $I$ in $[n]$. For any $z\in \mathbb{R}$, ${\rm sign}(z)$ denotes the sign function of $z$. For any $x\in \mathbb{R}^n$, denotes its $p$-norm as
$\|x\|_p = (\sum_{i=1}^p |x_i|^p)^{1/p}$. For simplicity, we denote $\|\cdot\|=\|\cdot\|_2$. For any vector $x\in \mathbb{R}^n$ and any index set $I\subseteq [n]$, denote $x_I$ as the subvector generated by the elements of $x$ indexed by $I$. For any closed and proper convex function $q:\mathbb{R}^n\rightarrow (-\infty,\infty]$, the proximal mapping of $q(\cdot)$ is defined by
\begin{align}
	{\rm Prox}_q(x):=\underset{y\in \mathbb{R}^n}{\arg\min}\ \Big\{ q(y)+\frac{1}{2}\|y-x\|^2\Big\},\label{eq: envelope}
\end{align}
for any $x\in \mathbb{R}^n$. It is known that ${\rm Prox}_q(\cdot)$ is Lipschitz continuous with modulus $1$ \cite{rockafellar1976monotone}.

\section{Dimension reduction via adaptive sieving}
\label{sec: screening}
Throughout this paper, we make the following blanket assumption, which is satisfied for many popular machine learning models, as discussed in \cite[Section 2.1]{zhou2017unified}.

\begin{assumption}
	Assume that the solution set of \eqref{eq: lasso_model}, denoted as $\Omega$, is nonempty and compact.
\end{assumption}

\subsection{The adaptive sieving strategy}
We propose an adaptive sieving strategy for solving sparse optimization models in the form of \eqref{eq: lasso_model}. An appealing property of this strategy is that it does not depend on the specific form of the regularizer and thus can be applied to sparse optimization problems with general regularization functions. More importantly, due to the adaptive nature of the AS strategy, it can sieve out a very large proportion of inactive features to significantly reduce the dimension of the problems and accelerate the computation by a large margin. This further shows that the AS strategy actually serves as a powerful dimension reduction technique for sparse optimization models.

\begin{algorithm}[H]
	\caption{An adaptive sieving strategy for solving \eqref{eq: lasso_model}}
	\label{alg:screening}
	\begin{algorithmic}[1]
		\STATE \textbf{Input}: an initial index set $I^0 \subseteq [n]$, a given tolerance $\epsilon\geq 0$ and a given positive integer $k_{\max}$ (e.g., $k_{\max}=500 $).
		\STATE \textbf{Output}: an approximate solution $x^*$ to the problem \eqref{eq: lasso_model} satisfying $\|R(x^*)\|\leq \epsilon$.
		\STATE \textbf{1}. Find
		\begin{align}
			x^{0} \in \underset{ x\in \mathbb{R}^n} {\arg\min} \  \Big\{\Phi(x)  +P(x) - \langle \delta^0,x\rangle	\ \mid \ x_{\overline{I^0}}=0\Big\},\label{eq: lambda0_problem}
		\end{align}
		where $\delta^0\in \mathbb{R}^n$ is an error vector such that $\|\delta^0\|\leq \epsilon$ and $(\delta^0)_{\overline{I^0}}=0$.
		
		\textbf{2}. Compute $R(x^0)$ and set $s=0$.
		\WHILE{$\|R(x^{s})\|> \epsilon$}
		\STATE \textbf{3.1}. Create $J^{s+1}$ as
		\begin{align}
			J^{s+1} = \Big\{ j\in \overline{I^s}\; \mid \; (R(x^s))_j\neq 0\Big\}.\label{eq: create_J}
		\end{align}
		(We prove in Theorem~\ref{thm: convergence_whileloop} that $J^{s+1}\not=\emptyset$.)
		Let $k$ be a positive integer satisfying $k \leq \min \{ |J^{s+1}|, k_{\max}\}$ and define 
		\begin{align*}
			\widehat{J}^{s+1} = \Big\{j \in J^{s+1} ~|~ \mbox{ $|(R(x^s))_j|$ is among the first $k$ largest values in $\{ |(R(x^s))_i| \}_{i \in J^{s+1} } $ } \Big\}.
		\end{align*}
		Update $I^{s+1}$ as:		
		\begin{align*}
			I^{s+1} \leftarrow I^{s} \cup \widehat{J}^{s+1}.
		\end{align*}
		\STATE \textbf{3.2}. Solve the following constrained problem with the initialization $x^s$,
		\begin{align}
			x^{s+1} \in \underset{ x\in \mathbb{R}^n} {\arg\min} \  \Big\{\Phi(x)  + P(x) - \langle \delta^{s+1},x\rangle \ \mid \ x_{\overline{I^{s+1}}}=0\Big\},\label{eq: constrained}
		\end{align}
	     where $\delta^{s+1}\in \mathbb{R}^n$ is an error vector such that $\|\delta^{s+1}\|\leq \epsilon$ and $(\delta^{s+1})_{\overline{I^{s+1}}}=0$.
		\STATE \textbf{3.3}: Compute $R(x^{s+1})$ and set $s\leftarrow s+1$.
		\ENDWHILE
		\STATE \textbf{return}: Set $x^*=x^{s}$.
	\end{algorithmic}
\end{algorithm}

We give the details of the AS strategy in Algorithm \ref{alg:screening}. Here in order to measure the accuracy of the approximate solutions, we define the proximal residual function $R:\mathbb{R}^n\rightarrow \mathbb{R}^n$ associated with the problem \eqref{eq: lasso_model} as
\begin{align}
	R(x):=x-{\rm Prox}_{P}(x-\nabla \Phi(x)),\quad  x\in \mathbb{R}^n.\label{eq:residual}
\end{align}
The KKT condition of \eqref{eq: lasso_model} implies that $\bar{x}\in \Omega$ if and only if $R(\bar{x})=0$.

A few remarks of Algorithm \ref{alg:screening} are in order. First, in Algorithm \ref{alg:screening}, the introduction of the error vectors $\delta^0$, $\{\delta^{s+1}\}$ in \eqref{eq: lambda0_problem} and \eqref{eq: constrained} implies  that the corresponding minimization problems can be solved inexactly. It is important to emphasize that the vectors are not a priori given but they are the errors incurred when the original problems (with $\delta^0=0$ in \eqref{eq: lambda0_problem} and $\delta^{s+1}=0$ in \eqref{eq: constrained}) are solved inexactly. We will explain how the error vectors $\delta^0$, $\{\delta^{s+1}\}$ in \eqref{eq: lambda0_problem} and \eqref{eq: constrained} can be obtained in Section \ref{sec: analysis}. Second, the initial active feature index set $I^0$ is suggested to be chosen such that $|I^0|\ll n$, in consideration of the computational cost. Third, the sizes of the reduced problems \eqref{eq: lambda0_problem} and \eqref{eq: constrained} are usually much smaller than $n$, which will be demonstrated in the numerical experiments.

We take the least squares linear regression problem as an example to show an efficient initialization of the index set $I^0$ via the correlation test, which is similar to the idea of the surely independent screening rule in \cite{fan2008sure}. Consider $\Phi(x) = \frac{1}{2}\|Ax-b\|^2$, where $A=[a_1,a_2,\cdots,a_n] \in \mathbb{R}^{m \times n}$ is a given feature matrix, $b\in \mathbb{R}^m$ is a given response vector. One can choose $k \lceil \sqrt{n} \rceil $ initial active features based on the correlation test between each feature vector $a_i$ and the response vector $b$. That is, one can compute $s_i := |\langle a_i, b \rangle|/(\|a_i\|\|b\|)$ for $i=1,\cdots,n$, and choose the initial guess of $I^0$ as
\begin{align*}
	I^0=\left\{i\in [n]: s_i\mbox{ is among the first } k \lceil \sqrt{n} \rceil \mbox{ largest values in} \; s_1,\ldots,s_n \right\}.
\end{align*}
In practice, we usually choose $k = 10$.

\subsection{Examples of the regularizer}
\label{sec: example-regularizers}
As we shall see in Algorithm \ref{alg:screening}, there is no restriction on the form of the regularizer $P(\cdot)$. In particular, it is not necessary to be separable or positively homogeneous, which is an important generalization compared to the existing screening rules \cite{tibshirani2012strong,ghaoui2010safe,wang2013lasso,wang2015lasso,xiang2016screening,zeng2021hybrid,zeng2017biglasso}. The only requirement for $P(\cdot)$ is that its proximal mapping ${\rm Prox}_P(\cdot)$ can be computed in an efficient way. Fortunately, it is the case for almost all popular regularizers used in practice. Below, we list some examples of the regularizers for better illustration.
\begin{itemize}
	\item Lasso regularizer \cite{tibshirani1996regression}:
	\begin{align*}
		P(x) = \lambda \|x\|_1,\quad x\in \mathbb{R}^n,
	\end{align*}
	where $\lambda > 0$ is a given parameter.
	\item Elastic net regularizer \cite{zou2005regularization}:
	\begin{align*}
		P(x) = \lambda_1 \|x\|_1+\lambda_2 \|x\|^2,\quad x\in \mathbb{R}^n,
	\end{align*}
	where $\lambda_1,\lambda_2 > 0$ are given parameters.
	\item Sparse group lasso regularizer \cite{eldar2009robust,jacob2009group}:
	\begin{align*}
		P(x) = \lambda_1 \|x\|_1 + \lambda_2 \sum_{l=1}^g w_l \|x_{G_l}\|,\quad x\in \mathbb{R}^n,
	\end{align*}
	where $\lambda_1,\lambda_2>0$, $w_1,\cdots,w_g\geq 0$ are parameters, and $\{G_1,\cdots,G_g\}$ is a disjoint partition of the set $[n]$. For the limiting case when $\lambda_1=0$, we get the group lasso regularizer \cite{yuan2006model}.
	\item Exclusive lasso regularizer \cite{zhou2010exclusive,kowalski2009sparse}:
	\begin{align*}
		P(x) = \lambda \sum_{l=1}^g  \|w_{G_l} \circ x_{G_l}\|_1^2,\quad x\in \mathbb{R}^n,
	\end{align*}
	where $\lambda > 0$ is a given parameter, $w\in \mathbb{R}^n_{++}$ is a weight vector and $\{G_1,\cdots,G_g\}$ is a disjoint partition of the index set $[n]$.
	\item Sorted L-One Penalized Estimation (SLOPE) \cite{bogdan2015slope}:
	\begin{align*}
		P(x) = \sum_{i=1}^n \lambda_i |x|_{(i)},
	\end{align*}
	with parameters $\lambda_1\geq \cdots \geq \lambda_n\geq 0$ and $\lambda_1>0$. For a given vector $x\in \mathbb{R}^n$, we denote $|x|$ to be the vector in $\mathbb{R}^n$ obtained from $x$ by taking the absolute value of its components. We define $|x|_{(i)}$ to be the $i$-th largest component of $|x|$ such that $|x|_{(1)}\geq \cdots \geq |x|_{(n)}$.
\end{itemize}

Among the above examples, the screening rules of the Lasso regularizer and the group lasso regularizer have been intensively studied in \cite{tibshirani2012strong,ghaoui2010safe,wang2013lasso,wang2015lasso,xiang2016screening,zeng2021hybrid,zeng2017biglasso}. Recently, a strong screening rule has been proposed for SLOPE \cite{SSR-SLOPE}. However, no unified approach has been proposed for sparse optimization problems with general regularizers. Fortunately, our proposed AS strategy is applicable to all the above examples, which includes three different kinds of ``challenging" regularizers. In particular, 1) the sparse group lasso regularizer is a combination of two regularizers; 2) the SLOPE is not separable; 3) the exclusive lasso regularizer is not positively homogeneous.

\section{Theoretical analysis of the AS strategy}
\label{sec: analysis}
In this section, we provide the theoretical analysis of the proposed AS strategy for solving sparse optimization problems presented in Algorithm \ref{alg:screening}. An appealing advantage of our proposed AS strategy is that it allows the involved reduced problems to be solved inexactly due to the introduction of the error vectors $\delta^0$, $\{\delta^{s+1}\}$ in \eqref{eq: lambda0_problem} and \eqref{eq: constrained}. It is important to emphasize that the vectors are not a priori given but they are the errors incurred when the original problems (with $\delta^0=0$ in \eqref{eq: lambda0_problem} and $\delta^{s+1}=0$ in \eqref{eq: constrained}) are solved inexactly. The following proposition explains how the error vectors $\delta^0$, $\{\delta^{s+1}\}$ in \eqref{eq: lambda0_problem} and \eqref{eq: constrained} can be obtained.

\begin{proposition}
	Given any $s = 0,1,\cdots$. The updating rule of $x^s$ in Algorithm \ref{alg:screening} can be interpreted in the procedure as follows. Let $M_s$ be a linear map from $\mathbb{R}^{|I^s|}$ to $\mathbb{R}^n$ defined as
	\begin{align*}
		(M_s z)_{I^s} = z,\quad (M_s z)_{\overline{I^s}} = 0,\quad z\in \mathbb{R}^{|I^s|},
	\end{align*}
	and $\Phi^s$, $P^s$ be functions from $\mathbb{R}^{|I^s|}$ to $\mathbb{R}$ defined as $\Phi^s(z):= \Phi(M_s z)$, $P^s(z) := P(M_s z)$ for all $z\in \mathbb{R}^{|I^s|}$. Then $x^s\in \mathbb{R}^n$ can be computed as
	\begin{align*}
		(x^s)_{I^s}:={\rm Prox}_{P^s}(\hat{z}-\nabla \Phi^s(\hat{z})),
	\end{align*}
	and $(x^s)_{\overline{I^s}}=0$, where $\hat{z}$ is an approximate solution to the problem
	\begin{align}
		\min_{z\in \mathbb{R}^{|I^s|}}\ \Big\{ \Phi^s(z) + P^s(z)\Big\},\label{eq: reduced_0}
	\end{align}
	which satisfies
	\begin{align}
		\|\hat{\delta}\|\leq \epsilon,\quad \hat{\delta}:=\hat{z}-(x^s)_{I^s} +\nabla \Phi^s((x^s)_{I^s})-\nabla \Phi^s(\hat{z}), \label{eq: control_delta0}
	\end{align}
	where $\epsilon$ is the parameter given in Algorithm \ref{alg:screening}.
	
	If the function $\Phi(\cdot)$ is $L$-smooth, that is, it is continuously differentiable and its gradient is Lipschitz continuous with constant $L$:
	\begin{align*}
		\|\nabla \Phi(x)-\nabla \Phi(y)\|\leq L \|x-y\|,\quad \forall x,y\in \mathbb{R}^n,
	\end{align*}
	then the condition \eqref{eq: control_delta0} can be achieved by
	\begin{align*}
		\|\hat{z}-{\rm Prox}_{P^s}(\hat{z}-\nabla \Phi^s(\hat{z}))\|\leq \frac{\epsilon}{1+L}.
	\end{align*}
\end{proposition}
\begin{proof}
	Let $\{z^{i}\}$ be a sequence that converges to a solution of the problem \eqref{eq: reduced_0}. For any $i= 1,2,\cdots$, define $\varepsilon^{i}:=z^{i}-{\rm Prox}_{P^s}(z^{i}-\nabla \Phi^s(z^{i})) +\nabla \Phi^s({\rm Prox}_{P^s}(z^{i}-\nabla \Phi^s(z^{i})))-\nabla \Phi^s(z^{i})$. By the continuous differentiability of $\Phi^s(\cdot)$ and \cite[Lemma 4.5]{mengyu2015inexact}, we know that $\lim_{i\rightarrow \infty}\|\varepsilon^{i}\|=0$, which implies the existence of $\hat{z}$ in \eqref{eq: control_delta0}.
	
	Next we explain the reason why the updating rule of $x^s$ in Algorithm \ref{alg:screening} can be interpreted as the one stated in the proposition. Since $(x^s)_{I^s}:={\rm Prox}_{P^s}(\hat{z}-\nabla \Phi^s(\hat{z}))$, we have
	\begin{align*}
		\hat{z}-(x^s)_{I^s}-\nabla \Phi^s(\hat{z})\in  \partial P^s((x^s)_{I^s}).
	\end{align*}
	According to the definition of $\hat{\delta}$ in \eqref{eq: control_delta0}, we have
	\begin{align*}
		\hat{\delta} \in \nabla \Phi^s((x^s)_{I^s})+ \partial P^s((x^s)_{I^s}),
	\end{align*}
	which means that $x^s$ is the exact solution to the problem
	\begin{align*}
		\min_{ x\in \mathbb{R}^n} \  \Big\{\Phi(x)  + P(x) - \langle \delta^{s},x\rangle \ \mid \ x_{\overline{I^{s}}}=0\Big\},
	\end{align*}
	with $\delta^s:=M_s \hat{\delta}$. Moreover, we can see that $(\delta^s)_{\overline{I^s}}= (M_s \hat{\delta})_{\overline{I^s}}=0$ and
	\begin{align*}
		\|\delta^s\|= \|\hat{\delta}\|=\|\hat{z}-(x^s)_{I^s} +\nabla \Phi^s((x^s)_{I^s})-\nabla \Phi^s(\hat{z})\|\leq \epsilon.
	\end{align*}
	
	If $\Phi(\cdot)$ is $L$-smooth, we have that
	\begin{align*}
		\|\hat{\delta}\|&=\|\hat{z}-(x^s)_{I^s} +\nabla \Phi^s((x^s)_{I^s})-\nabla \Phi^s(\hat{z})\|\\
		&\leq \|\hat{z}-(x^s)_{I^s}\| +\|\nabla \Phi^s((x^s)_{I^s})-\nabla \Phi^s(\hat{z})\|\\
		& = \|\hat{z}-(x^s)_{I^s}\| +\|M_s^T \nabla \Phi(M_s (x^s)_{I^s})-M_s^T\nabla \Phi(M_s\hat{z})\|\\
		&\leq \|\hat{z}-(x^s)_{I^s}\| +\| \nabla \Phi(M_s (x^s)_{I^s})-\nabla \Phi(M_s\hat{z})\|\\
		&\leq \|\hat{z}-(x^s)_{I^s}\| +L\| M_s (x^s)_{I^s}-M_s\hat{z}\|\\
		& = (1+L)\|\hat{z}-(x^s)_{I^s}\|,
	\end{align*}
	which means that \eqref{eq: control_delta0} can be achieved by
	\begin{align*}
		\|\hat{z}-{\rm Prox}_{P^s}(\hat{z}-\nabla \Phi^s(\hat{z}))\|\leq \frac{\epsilon}{1+L}.
	\end{align*}
	This completes the proof.
\end{proof}

Next, we will show the convergence properties of Algorithm \ref{alg:screening} in the following proposition, which states that the algorithm will terminate after a finite number of iterations. 

\begin{theorem}\label{thm: convergence_whileloop}
	The while loop in Algorithm \ref{alg:screening} will terminate after a finite number of iterations.
\end{theorem}
\begin{proof}
	We first prove that when $\|R(x^{s})\|>\epsilon\geq 0$ for some $s\geq 0$, the index set $J^{s+1}$ defined in \eqref{eq: create_J} is nonempty. We prove this by contradiction. Suppose that $J^{s+1}=\emptyset$, which means
	\begin{align*}
		(R(x^s))_{\overline{I^{s}}}=0.
	\end{align*}
	According to the definition of $R(\cdot)$ in \eqref{eq:residual}, we have
	\begin{align*}
		R(x^s) - \nabla \Phi(x^s) \in \partial P(x^s-R(x^s)).
	\end{align*}
	Note that $x^{s}$ satisfies
	\begin{align*}
		x^{s} \in \underset{ x\in \mathbb{R}^n} {\arg\min} \  \Big\{\Phi(x)  + P(x) - \langle \delta^{s},x\rangle \ \mid \ x_{\overline{I^s}}=0\Big\},
	\end{align*}
	where $\delta^{s}\in \mathbb{R}^n$ is an error vector such that $(\delta^s)_{\overline{I^s}} = 0$ and $\|\delta^{s}\|\leq \epsilon$. By the KKT condition of the above minimization problem, we know that there exists a multiplier $y\in \mathbb{R}^{n}$ with $y_{I^s}=0$ such that
	\begin{align*}
		\left\{
		\begin{aligned}
			&0\in \nabla \Phi(x^{s}) + \partial P(x^{s})-\delta^{s} - y,\\
			& \Big( x^{s}\Big)_{\overline{I^s}}=0,
		\end{aligned}\right.
	\end{align*}
	which means
	\begin{align*}
		\delta^{s} + y - \nabla \Phi(x^{s})\in \partial P(x^{s}).
	\end{align*}
	By the maximal monotonicity of the operator $\partial P$, we have
	\begin{align*}
		\langle R(x^s) -\delta^s -y ,-R(x^s)\rangle \geq 0,
	\end{align*}
	which implies that
	\begin{align*}
		\|R(x^s)\|^2\leq \langle \delta^s + y,R(x^s)\rangle = \langle (\delta^s)_{I^s}, (R(x^s))_{I^s}\rangle\leq \|\delta^s\| \|R(x^s)\|\leq \epsilon \|R(x^s)\|.
	\end{align*}
	Thus, we get $\|R(x^s)\|\leq \epsilon$, which is a contradiction.
	
	Therefore, we have that for any $s\geq 0$, $J^{s+1}\neq \emptyset$ as long as $\|R(x^{s})\|>\epsilon$, which further implies that $\widehat{J}^{s+1}\neq \emptyset$. In other words, new indices will be added to the index set $I^{s+1}$ as long as the KKT residual has not achieved at the required accuracy. Since the total number of features $n$ is finite, the while loop in Algorithm \ref{alg:screening} will terminate after a finite number of iterations.
\end{proof}

Although we only have the finite termination guarantee for the proposed AS strategy, the superior empirical performance of the AS strategy will be demonstrated later in Section \ref{sec: experiments} with extensive numerical experiments. In particular, the number of AS iterations is no more than 5 (less than 3 for most of the cases) for solving a single sparse optimization problem in our experiments on both synthetic and real datasets.

\section{An efficient path generation method based on the AS strategy}
\label{sec: inexact}
When solving machine learning models in practice, we need to solve the model with a sequence of hyper-parameters and then select appropriate values for the hyper-parameters based on some methodologies, such as cross-validation. In this section, we are going to propose a path generation method based on the AS strategy to obtain solution paths of general sparse optimization models. Specifically, we will borrow the idea of the proposed AS strategy to solve the model \eqref{eq: lasso_model2} for a sequence of $\lambda$ in a highly efficient way.

Before presenting our path generation method, we first briefly discuss the ideas behind the two most popular screening rules for obtaining solution paths of various machine learning models, namely the strong screening rule (SSR) \cite{tibshirani2012strong} and the dual polytope projection based screening rule (DPP) \cite{wang2013lasso,wang2015lasso}. For convenience, we take the Lasso model as an illustration, with $\Phi(x) = \frac{1}{2}\|Ax - b\|^2$ and $P(x) = \|x\|_1$ in \eqref{eq: lasso_model2}, where $A=[a_1,a_2,\cdots,a_n] \in \mathbb{R}^{m\times n}$ is the feature matrix and $b\in \mathbb{R}^m$ is the response vector. Let $x^{*}(\lambda)$ be an optimal solution to \eqref{eq: lasso_model2}. The KKT condition implies that:
\begin{align*}
	a_i^T\theta^*(\lambda) \in \left\{
	\begin{array}{ll}
		\{ \lambda \ {\rm sign}(x^*(\lambda))\} & \mbox{if} \; x^*(\lambda)_i \not= 0 \\
		\left[-\lambda, \lambda\right] & \mbox{if} \; x^*(\lambda)_i = 0\\
	\end{array}
	\right.,
\end{align*}
where $\theta^*(\lambda)$ is the optimal solution of the associated dual problem:
\begin{align}
	\label{eq: dual_problem_lasso}
	\max_{\theta\in \mathbb{R}^m} \ \Big\{\frac{1}{2}\|b\|^2 - \frac{1}{2}\|\theta - b\|^2\; \mid \;|a_i^T\theta| \leq \lambda,\ i = 1, 2, \dots, n\Big\}.
\end{align}
The existing screening rules for the Lasso model are based on the fact that $x^*(\lambda)_i = 0$ if $|a_i^T\theta^*(\lambda)| < \lambda$. The difference is how to estimate $a_i^T\theta^*(\lambda)$ based on an optimal solution $x^*(\tilde{\lambda})$ of $({\rm P}_{\tilde{\lambda}})$ for some $\tilde{\lambda} > \lambda$, without solving the dual problem \eqref{eq: dual_problem_lasso}. The SSR \cite{tibshirani2012strong} discards the $i$-th predictor if
\begin{align*}
	|a_i^T\theta^*(\tilde{\lambda}) | \leq 2\lambda-\tilde{\lambda},
\end{align*}
by assuming the ``unit slope" bound condition: $
|a_i^T\theta^*(\lambda_1) - a_i^T\theta^*(\lambda_2)| \leq |\lambda_1-\lambda_2|$, for all $\lambda_1, \lambda_2 > 0$.
Since this assumption may fail, the SSR may screen out some active features by mistake. Also, the SSR only works for consecutive hyper-parameters with a small gap, since it requires $\lambda > \frac{\tilde{\lambda}}{2}$. Wang et al. \cite{wang2013lasso,wang2015lasso} proposed the DPP by carefully analysing the properties of the optimal solution to the dual problem \eqref{eq: dual_problem_lasso}. The key idea is, if we can estimate a region $\Theta_{\lambda}$ containing $\theta^*(\lambda)$, then
\begin{align*}
	\sup_{\theta \in \Theta_{\lambda}} |a_i^T\theta| < \lambda \quad \Longrightarrow \quad x^*(\lambda)_i = 0.
\end{align*}
They estimate the region $\Theta_{\lambda}$ by realizing that the optimal solution of \eqref{eq: dual_problem_lasso} is the  projection onto a polytope. As we can see, a tighter estimation of $\Theta_{\lambda}$ will induce a better safe screening rule.

The bottlenecks of applying the above screening rules to solve general machine learning models in the form of \eqref{eq: lasso_model} are: 1) the subdifferential of a general regularizer may be much more complicated than the Lasso regularizer, whose subdifferential is separable; 2) a general regularizer may not be positively homogeneous, which means that the optimal solution to the dual problem may not be the projection onto some convex set \cite[Theorem 13.2]{rockafellar1970convex}; 3) These screening rules may only exclude a small portion of zeros in practice even if the solutions of the problem are highly sparse, thus we still need to solve large-scale problems after the screening; 4) they implicitly require the reduced problems to be solved exactly, which is unrealistic for many machine learning models in practice.

\begin{algorithm}[ht]
	\caption{An AS-based path generation method for \eqref{eq: lasso_model2}}
	\label{alg:screening2}
	\begin{algorithmic}[1]
		\STATE \textbf{Input}: a sequence of hyper-parameters: $\lambda_0 > \lambda_1 > \dots > \lambda_k > 0$, and given tolerances  $\epsilon \geq 0$ and $\hat{\epsilon} \geq 0$.
		\STATE \textbf{Output}: a solution path: $x^*(\lambda_0), x^*(\lambda_1), x^*(\lambda_2), \dots, x^*(\lambda_k)$.
		\STATE \textbf{1}. Apply Algorithm \ref{alg:screening} to solve the problem $(\mbox{P}_{\lambda_0})$ with the tolerance $\epsilon$, and then denote the output as $x^*(\lambda_0)$. Let
		\begin{align*}
			I^*(\lambda_0) := \{j \;\mid\; |x^*(\lambda_0)_j| > \hat{\epsilon},\ j=1,\cdots,n\}.
		\end{align*}
		\FOR{$i = 1, 2, \dots, k$}
		\STATE \textbf{2.1}. Let $I^{0}(\lambda_i) = I^*(\lambda_{i-1})$.
		\STATE \textbf{2.2}. Apply Algorithm \ref{alg:screening} to solve the problem $(\mbox{P}_{\lambda_i})$ with the initial index set $I^{0}(\lambda_i)$, the initialization $x^*(\lambda_{i-1})$ and the tolerance $\epsilon$. Denote the output as $x^*(\lambda_i)$.
		\STATE \textbf{2.3}. Let
		\begin{align*}
			I^*(\lambda_i) := \{j \;\mid\; |x^*(\lambda_i)_j| > \hat{\epsilon},\ j=1,\cdots,n\}.
		\end{align*}
		\ENDFOR
	\end{algorithmic}
\end{algorithm}

To overcome these challenges, we propose a path generation method based on the AS strategy for solving the problem \eqref{eq: lasso_model2} presented in Algorithm \ref{alg:screening2}. Here we denote the proximal residual function $R_{\lambda}:\mathbb{R}^n\rightarrow \mathbb{R}^n$ associated with \eqref{eq: lasso_model2} as
\begin{align}
	R_{\lambda}(x):=x-{\rm Prox}_{\lambda P}(x-\nabla \Phi(x)),\quad  x\in \mathbb{R}^n.\label{eq:residual_lambda}
\end{align}
Assume that for any $\lambda > 0$, the solution set of \eqref{eq: lasso_model2} is nonempty and compact.

The design of the above algorithm directly leads us to the following theorem regarding its convergence property. The proof of the theorem can be obtained directly due to Theorem \ref{thm: convergence_whileloop}.

\begin{theorem}\label{thm: convergence_screening}
	The solution path $x^*(\lambda_0), x^*(\lambda_1),\dots, x^*(\lambda_k)$ generated by Algorithm \ref{alg:screening2} forms a sequence of approximate solutions to the problems $(\mbox{P}_{\lambda_0}), (\mbox{P}_{\lambda_1}),$ $\cdots, (\mbox{P}_{\lambda_k})$, in the sense that \[\|R_{\lambda_i}(x^*(\lambda_i))\|\leq \epsilon,\quad  i=0,1,\cdots,k.\]
\end{theorem}

In practice, we usually choose $\hat{\epsilon} = 10^{-10}$. We will compare the numerical performance of the proposed AS strategy to the popular screening rules in Section \ref{sec: experiments}, our experiment results will show that the AS strategy is much more efficient for solving large-scale sparse optimization problems.

\section{Numerical Experiments}
\label{sec: experiments}
In this section, we will present extensive numerical results to demonstrate the performance of the AS strategy \footnote{The code is available at \url{https://doi.org/10.5281/zenodo.15087424}}. We will test the performance of the AS strategy on the sparse optimization problems of the following form:
\begin{equation}
	\label{eq: model-numerical}
	\min_{x \in \mathbb{R}^n} \ \Big\{ h(Ax) + \lambda P(x)\Big\},
\end{equation}
where $A: \mathbb{R}^n \to \mathbb{R}^m$ is a given data matrix, $h: \mathbb{R}^m \to \mathbb{R}$ is a convex, twice continuously differentiable loss function, $P: \mathbb{R}^n \!\to\! (-\infty, +\infty]$ is a proper closed convex regularization function, and $\lambda > 0$ is the hyper-parameter. More specifically, we consider the two most commonly used loss functions for $h(\cdot)$:
\begin{itemize}
	\item[(1)] (linear regression) $h(y)=\sum_{i=1}^m (y_i-b_i)^2/2$, for some given vector $b\in \mathbb{R}^m$;
	\item[(2)] (logistic regression) $h(y) = \sum_{i=1}^m \log(1 + \exp(-b_i y_i))$, for some given vector $b\in \{-1,1\}^m$.
\end{itemize}
For the regularization function $P(\cdot)$, we consider four important regularizers in statistical learning models: Lasso, sparse group lasso, exclusive lasso, and SLOPE. Details of the regularizers can be found in Section \ref{sec: example-regularizers}. The algorithms used for solving each sparse optimization model and its corresponding reduced problems when applying the AS strategy will be specified later in the related subsections. 

In our experiments, we measure the accuracy of the obtained solution by the relative KKT residual:
\begin{align}
	\eta_{\rm KKT} := \frac{\|x -  {\rm Prox}_{\lambda P}(x - A^T\nabla h(Ax))\|}{1 + \|x\| + \|A^T\nabla h(Ax)\|}, \label{eq: eta_kkt}
\end{align}
as suggested in \cite{li2018highly}. Here $\eta_{\rm KKT}$ measures how close the optimality condition of \eqref{eq: model-numerical} is to being met, which is particularly crucial in scenarios where the function values decrease slowly but significant progress towards satisfying the optimality condition is being made. We terminate the tested algorithm when $\eta_{\rm KKT}\leq \varepsilon$, where $\varepsilon > 0$ is a given tolerance, which is set to be $10^{-6}$ by default. We choose $\hat{\epsilon} = 10^{-10}$ in Algorithm \ref{alg:screening2}. All our numerical results are obtained by running MATLAB(2022a version) on a Windows workstation  (Intel Xeon E5-2680 @2.50GHz).

The numerical experiments will be organized as follows. First, we present the numerical performance of the AS strategy on the mentioned models on synthetic data sets in Section \ref{sec:linear_synthetic} and Section \ref{sec:logistic_synthetic}. To better demonstrate the effectiveness and flexibility of our AS strategy, we compare it with other existing approaches for the path generation in Section \ref{sec: allscreeningrules}, and test its performance with its reduced problems solved by different algorithms in Section \ref{sec: algo_com}. Finally, we present the numerical results on nine real data sets in Section \ref{sec:real}. 

\subsection{Performance of the AS strategy for linear regression on synthetic data}
\label{sec:linear_synthetic}
In this subsection, we conduct numerical experiments to demonstrate the performance of the AS strategy for linear regression on synthetic data. 

We first describe the generation of the synthetic data sets. We consider the models with group structure (e.g., sparse group lasso and exclusive lasso) and without group structure (e.g., Lasso and SLOPE). For simplicity, we randomly generate data with groups and test all the models on these data. Motivated by \cite{campbell2017within}, we generate the synthetic data using the model $b = Ax ^* +\xi$, where $x ^*$ is the predefined true solution and $\xi \sim \mathcal{N}(0, I_m)$ is a random noise vector. Given the number of observations $m$, the number of groups $g$ and the number of features $p$ in each group, we generate each row of the matrix $A \in \mathbb{R}^{m \times gp}$ by independently sampling a vector from a multivariate normal distribution $\mathcal{N}(0, \Sigma)$, where $\Sigma$ is a Toeplitz covariance matrix with entries $\Sigma_{ij} = 0.9^{|i - j|}$ for the features in the same group, and $\Sigma_{ij} = 0.3^{|i - j|}$ for the features in different groups. For the ground-truth $x^{*}$, we randomly generate $r$ nonzero elements in each group with i.i.d values drawn from the uniform distribution on $[0,10]$. In particular, we test the case where $r = \lfloor 0.1\% \times p \rfloor$. More experiments for the case where $r = \lfloor 1\% \times p \rfloor$ are given in Appendix \ref{sec: other_sparsity}.

Here, we mainly focus on solving the regression models in high-dimensional settings. In our experiments, we set $m \in \{500, 1000, 2000\}$. For the number of features, for simplicity, we fix $g$ to be $20$, but vary the number of features $p$ in each group from $5000$ to $60000$. That is, we vary the total number of features $n=gp$ from $100000$ to $1200000$.

\subsubsection{Numerical results on Lasso linear regression problems}
\label{sec: lasso-linear-synthetic}

In this subsection, we present the numerical results of the Lasso linear regression model. We will apply the state-of-the-art semismooth Newton based augmented Lagrangian (SSNAL) method\footnote{\url{https://github.com/MatOpt/SuiteLasso}} to solve the Lasso model and its corresponding reduced problems generated by the AS strategy. The comparison of the efficiency between the SSNAL method and other popular algorithms for solving Lasso models can be found in \cite{li2018highly}. Following the numerical experiment settings in \cite{li2018highly}, we test the numerical performance of the AS strategy for generating a solution path for the Lasso linear regression model with $\lambda = \lambda_c\|A^T b\|_{\infty}  $, where $\lambda_c$ is taken from $10^{-1}$ to $10^{-4}$ with 20 equally divided grid points on the $\log_{10}$ scale.

We show the numerical results in Figure \ref{fig: lasso_linear_allsize}, where the performance of the AS strategy and the well-known warm-start technique for generating solution paths of Lasso linear regression problems are compared with different problem sizes $(m,n)$. From the numerical results, we can see that the AS strategy can significantly reduce the computational time of path generation. In particular, the AS strategy can accelerate the SSNAL method up to $\mathbf{43}$ times for generating solution paths for the Lasso linear regression model. Moreover, the AS strategy scales efficiently as the number of features increases, which makes it especially effective for handling high-dimensional data.

\begin{figure}[H]
	\centering
	\includegraphics[width = 1\columnwidth]{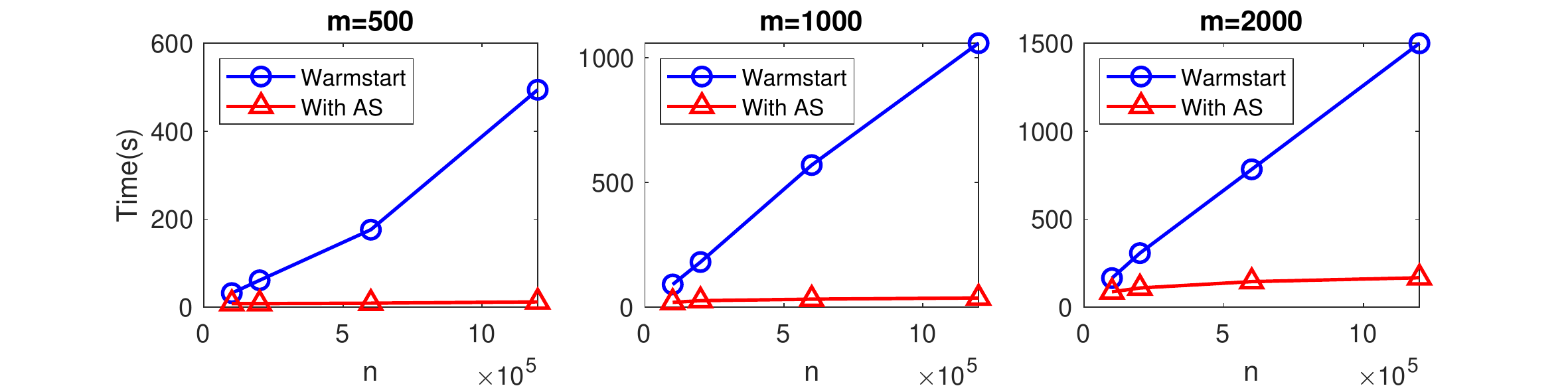}
	\caption{Performance of the AS strategy applied to the generation of solution paths for the Lasso linear regression model on synthetic data sets with different problem sizes.}
	\label{fig: lasso_linear_allsize}
\end{figure}

\begin{figure}[H]
	\centering
	\subfigure[Selected active features]{
		\includegraphics[width = 0.65\columnwidth]{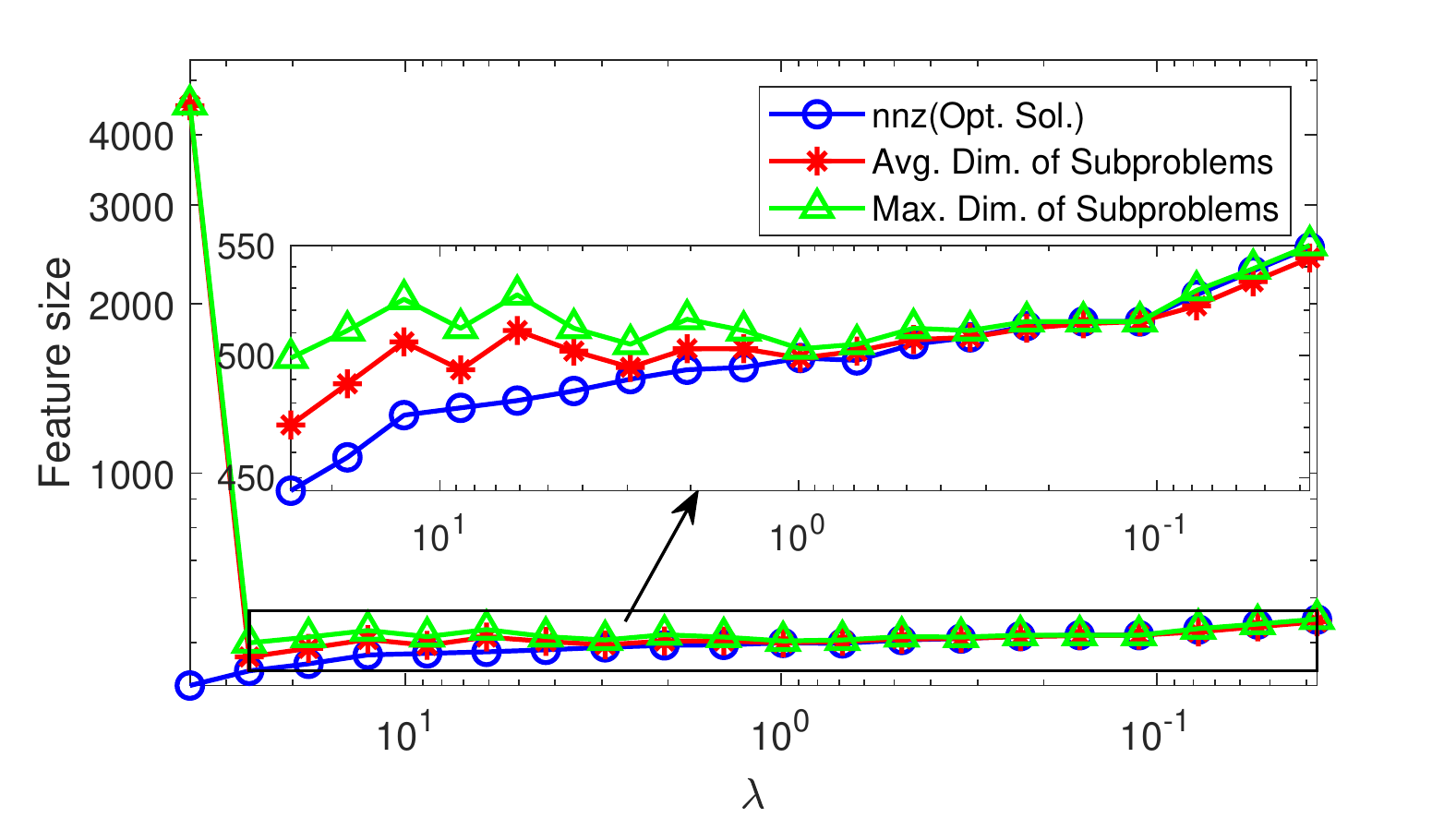}\label{fig: lasso_linear_AS1_1}}
	\hspace{-0.5cm}
	\subfigure[Number of sieving rounds]{\raisebox{5pt}{\includegraphics[width = 0.28\columnwidth]{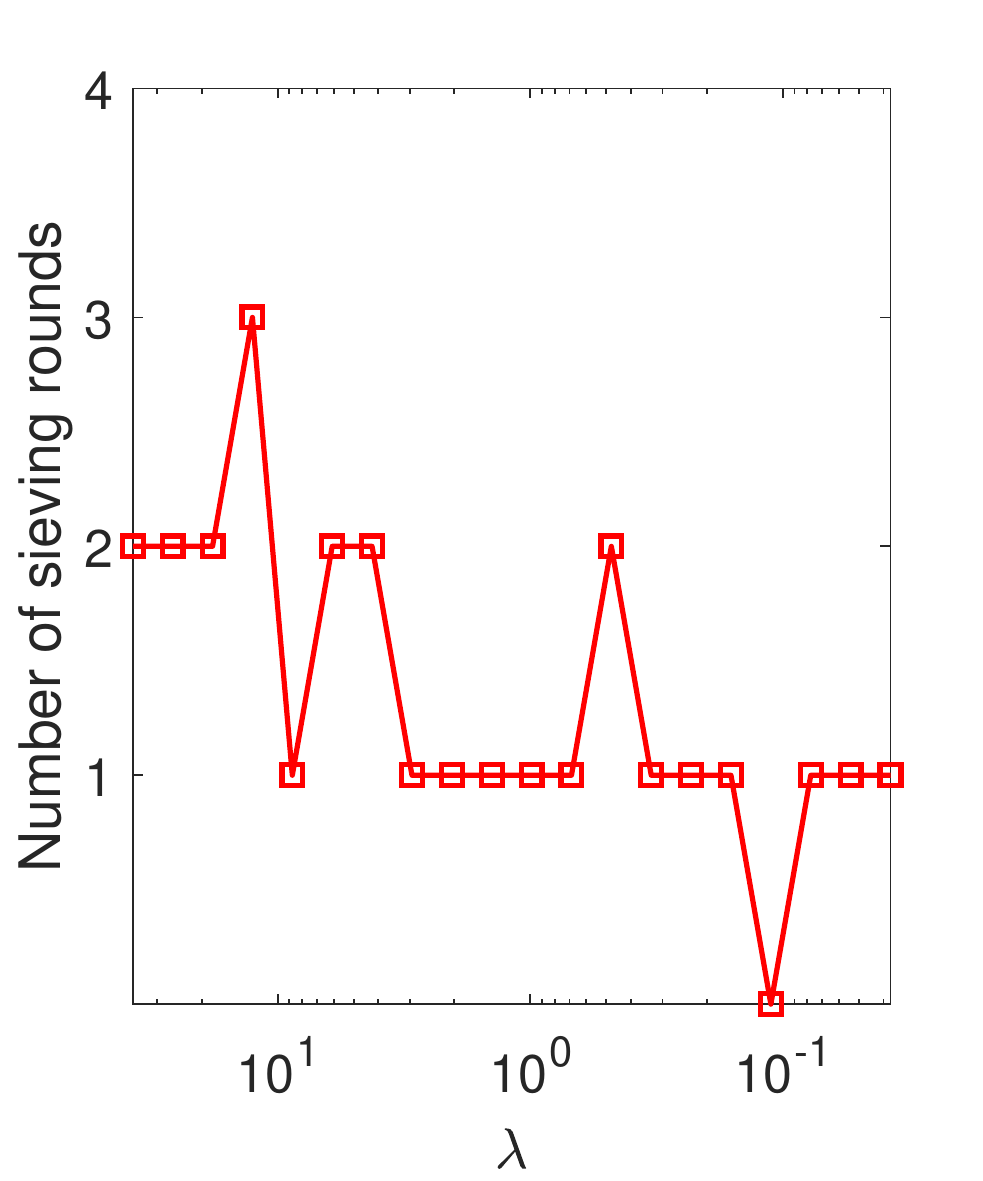}}\label{fig: lasso_linear_AS1_2}}
	\caption{Performance profile of the AS strategy on the Lasso linear regression model for the case when $m=500,n=100,000$.}
	\label{fig: lasso_linear_AS1}
\end{figure}

In order to further demonstrate the power of the AS strategy, in Figure \ref{fig: lasso_linear_AS1_1}, we plot the average (and maximum) problem size of the reduced problems in the AS strategy against the number of nonzero entries of the optimal solutions in the solution path for the case when $m=500,n=100,000$. We can see that the average reduced problem sizes in the AS strategy can nearly match the actual number of nonzeros in the optimal solutions, except for the first problem with the largest $\lambda$. This clearly demonstrates the dimensionality reduction achieved by the AS strategy when solving sparse optimization problems, especially in high-dimensional settings. In addition, Figure \ref{fig: lasso_linear_AS1_2} shows that we only need no more than three rounds of sieving in the AS strategy to solve the problem for each $\lambda$, which demonstrates that our sieving technique is quite effective in practice.

\subsubsection{Numerical results on exclusive lasso linear regression problems}
\label{sec: exclusive-lasso-linear-synthetic}
In this subsection, we present the numerical results of the exclusive lasso linear regression model. We apply the state-of-the-art dual Newton based proximal point algorithm (PPDNA) \cite{lin2023ahighly} to solve the exclusive lasso model and its corresponding reduced problems generated by the AS strategy. The comparison of the efficiency between the PPDNA and other popular algorithms for solving exclusive lasso models can be found in \cite{lin2023ahighly}. We follow the numerical experiment settings in \cite{lin2023ahighly} and test the numerical performance of generating a solution path for the exclusive lasso regression model with $\lambda = \lambda_c\|A^T b\|_{\infty}$ and $w$ being the vector of all ones. Here, $\lambda_c$ is taken from $10^{-1}$ to $10^{-4}$ with 20 equally divided grid points on the $\log_{10}$ scale. 

\begin{figure}[H]
	\centering
	\includegraphics[width = 1\columnwidth]{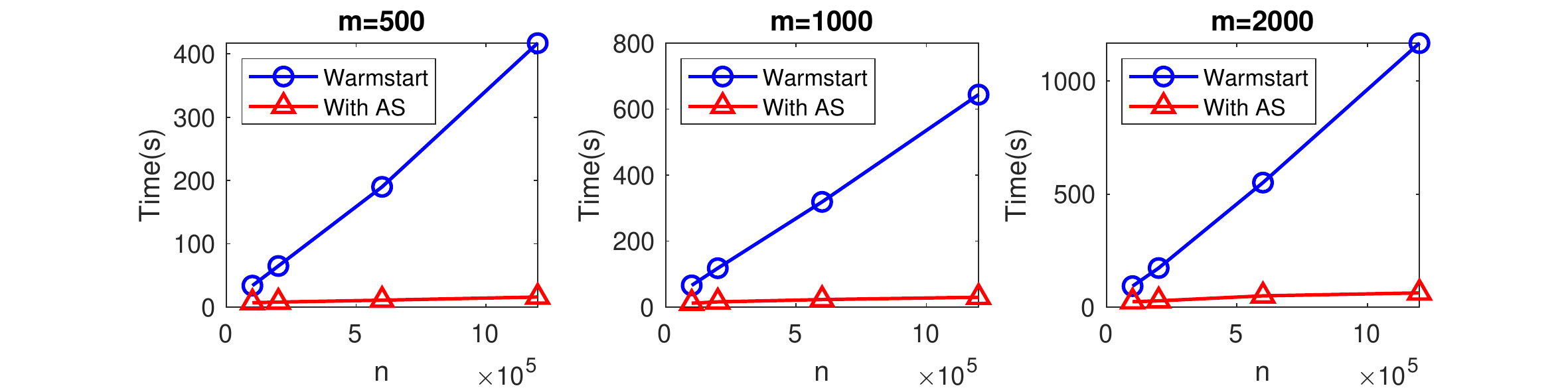}
	\caption{Performance of the AS strategy applied to the generation of solution paths for the exclusive lasso linear regression model on synthetic data sets with different problem sizes.}
	\label{fig: exlasso_linear_allsize}
\end{figure}

We summarize the numerical results on the path generation of the exclusive lasso linear regression model with different problem sizes in Figure \ref{fig: exlasso_linear_allsize}. We can see that the AS strategy gives excellent performance in generating the solution paths for the exclusive lasso linear regression models, in the sense that it accelerates the PPDNA algorithm up to \textbf{30} times. In addition, the average (and maximum) problem sizes of the reduced problems against the number of nonzeros in the optimal solutions for the case when $m=500,n=100,000$ is shown in Figure \ref{fig: exlasso_linear_AS1_1}. The number of sieving rounds for this case is also provided in Figure \ref{fig: exlasso_linear_AS1_2}. We can see that AS is able to highly reduce the dimensions of the optimization problems along the solution paths.

More experiments on the numerical performance of the AS strategy when solving the linear regression models with sparse group lasso regularizer and SLOPE regularzier are shown in Appendix \ref{sec: sparse-group-lasso-linear-regression-synthetic} and Appendix \ref{sec: SLOPE-linear-regression-synthetic}, respectively.

\begin{figure}[H]
	\subfigure[Selected active features]{
		\includegraphics[width = 0.65\columnwidth]{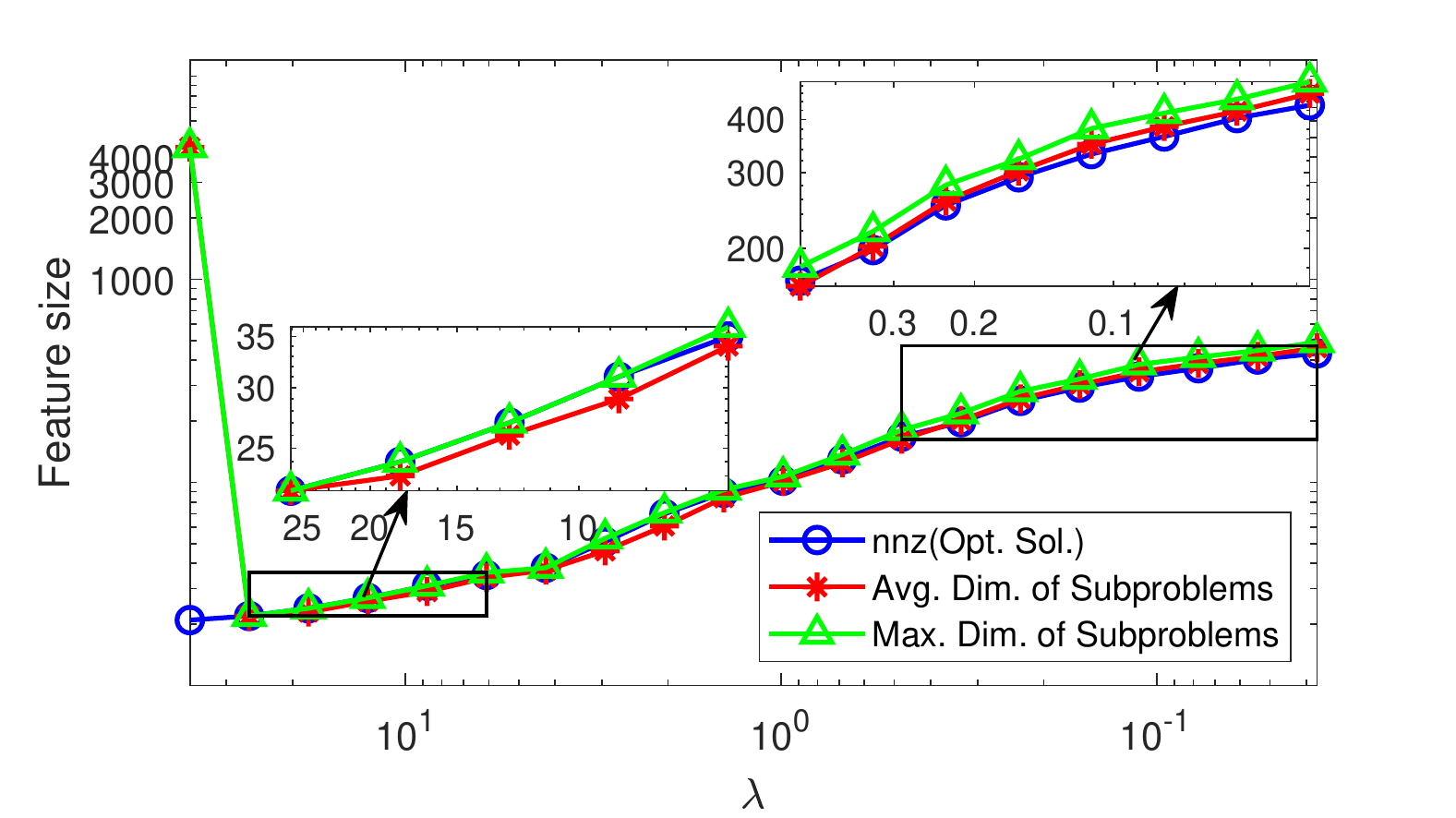}\label{fig: exlasso_linear_AS1_1}}
	\hspace{-0.5cm}
	\subfigure[Number of sieving rounds]{\raisebox{5pt}{\includegraphics[width = 0.28\columnwidth]{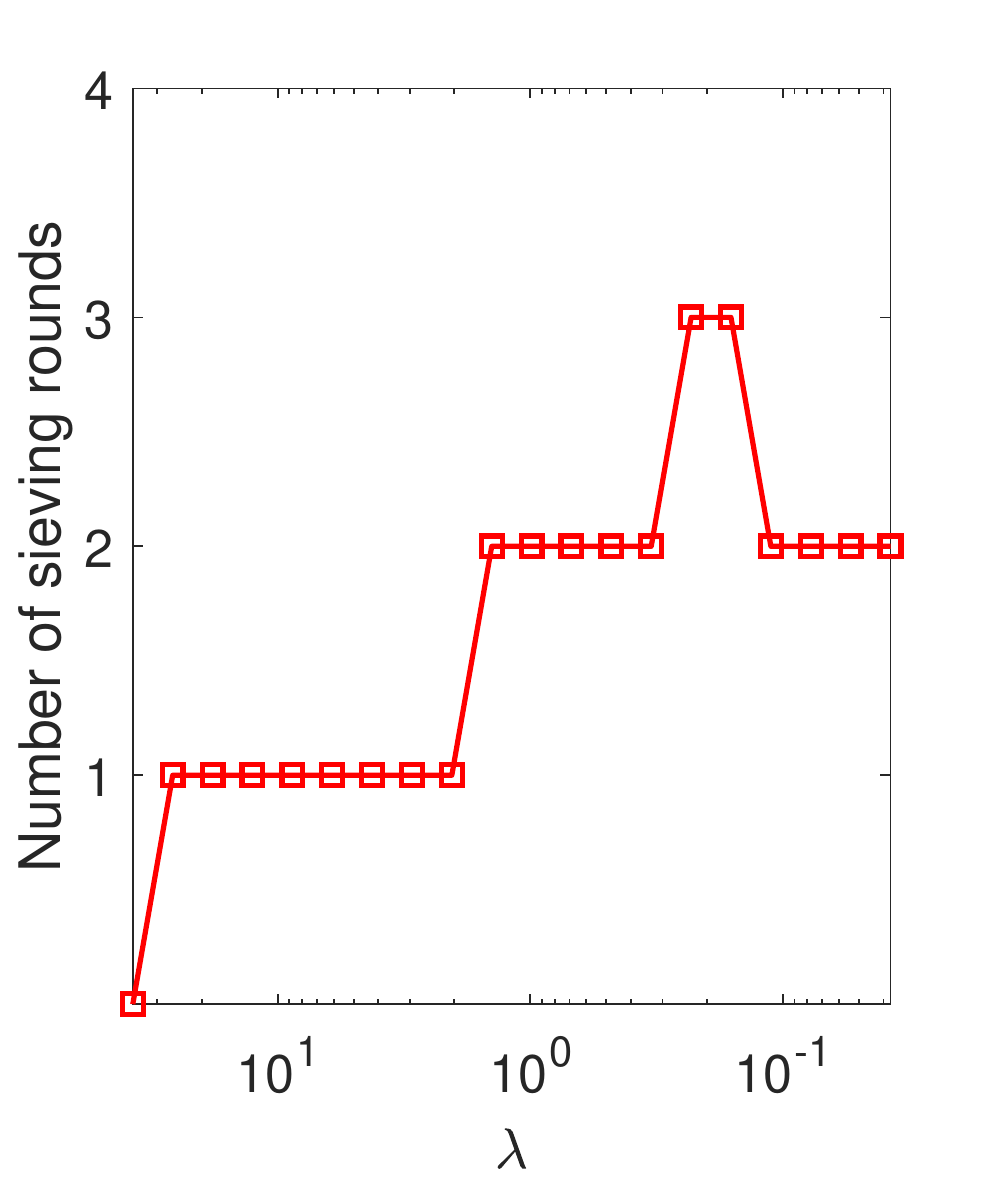}}\label{fig: exlasso_linear_AS1_2}}
	\caption{Performance profile of the AS strategy on the exclusive lasso linear regression model for the case when $m=500,n=100,000$.}
	\label{fig: exlasso_linear_AS1}
\end{figure}

\subsection{Performance of the AS strategy for logistic regression on synthetic data}
\label{sec:logistic_synthetic}

To test the regularized logistic regression problem, we generate $A$ and $x^*$ following the same settings as in Section \ref{sec:linear_synthetic} and define $b_i=1$ if $Ax^* +\tilde{\xi}\geq 0$, and $-1$ otherwise, where $\tilde{\xi} \sim \mathcal{N}(0, I_m)$.

\begin{table}[H]
	\caption{Numerical performance of the AS strategy applied to the generation of 
		solution paths (20 different $\lambda$'s) for the Lasso logistic regression model on synthetic data sets. In the table, `S. Rnd' represents the total number of the AS rounds during the path generation, `Avg. D.' (`Max. D.') denotes the average (maximum) dimension of the reduced problems on the solution path.}
	\label{tab: AS-PPDNA-logistic-lasso-Tab1}
	\renewcommand\arraystretch{1.2}
	\centering
	\begin{tabular}{|c|r|c @{\ $|$\ } c|c|c @{\ $|$\ } c|}  
		\hline
		& & \multicolumn{2}{c|}{Total time (hh:mm:ss)}  & \multicolumn{3}{c|}{Information of the AS} \\
		\hline
		$m$  & $n (= g \times p)$  & Warmstart & With AS  & S. Rnd & Avg. D. & Max. D. \\
		\hline
		\multirow{4}*{\tabincell{c}{$500$}}
		& $100,000$ & 00:02:30 & 00:00:06 & 23  & 564 & 3170 \\
		& $200,000$ & 00:04:18 & 00:00:05 & 18  & 622 & 4473 \\
		& $600,000$ & 00:11:40 & 00:00:08 & 19  & 764 & 7752 \\
		& $1,200,000$ & 00:23:06 & 00:00:11 & 20  & 668 & 10951 \\
		\hline
		\multirow{4}*{\tabincell{c}{$1000$}}
		& $100,000$ & 00:05:55 & 00:00:12 & 23  & 843 & 3227 \\
		& $200,000$ & 00:10:12 & 00:00:14 & 22  & 949 & 4512 \\
		& $600,000$ & 00:28:29 & 00:00:21 & 20  & 1119 & 7763 \\
		& $1,200,000$ & 00:53:27 & 00:00:28 & 18  & 1335 & 10954 \\
		\hline
		\multirow{4}*{\tabincell{c}{$2000$}}
		& $100,000$ & 00:17:27 & 00:01:12 & 26  & 1237 & 3251 \\
		& $200,000$ &  00:26:19 & 00:01:15 & 25  & 1553 & 4656 \\
		& $600,000$ &  01:09:52 & 00:01:23 & 22  & 1956 & 7859 \\
		& $1,200,000$ & 02:01:53 & 00:01:42 & 19  & 2282 & 11010 \\
		\hline
	\end{tabular}
\end{table}

\subsubsection{Numerical results on Lasso logistic regression problems}

We present the results on the Lasso logistic regression model, where the (reduced) problems are again solved by the SSNAL method \cite{li2018highly}. We test the performance on generating a solution path for the Lasso logistic regression model with $\lambda =\lambda_c \|A^T b\|_{\infty} $, where $\lambda_c$ is taken from $10^{-1}$ to $10^{-4}$ with 20 equally divided grid points on the $\log_{10}$ scale. The detailed numerical results are shown in Table \ref{tab: AS-PPDNA-logistic-lasso-Tab1}, where the performance of the AS strategy and the well-known warm-start technique for generating solution paths of Lasso logistic regression problems are compared on the problems with different sizes $(m,n)$. For better illustration of the AS performance, we also list the total number of sieving rounds, together with the average (maximum) dimension of the reduced problems on the solution path.

We can see from the table that the AS strategy also works extremely well for solving the Lasso logistic regression problems, in the sense that it reduces the problem size by a large margin for each case and highly accelerates the computation of the path generation.

\subsubsection{Numerical results on exclusive lasso logistic regression problems}
In this subsection, we present the numerical results on the exclusive lasso logistic regression model, where the (reduced) problems are solved by the PPDNA \cite{lin2023ahighly}. We test the numerical performance of the AS strategy for generating a solution path for the exclusive lasso logistic regression model with uniform weights and $\lambda = \lambda_c \|A^T b\|_{\infty}$, where $\lambda_c$ is taken from $10^{-1}$ to $10^{-4}$ with 20 equally divided grid points on the $\log_{10}$ scale. 

The results are presented in Table \ref{tab: AS-PPDNA-logistic-exclusive-lasso-Tab1}. As demonstrated, our AS strategy once again exhibits excellent performance in significantly reducing the problem sizes during the path generation of the exclusive lasso logistic regression problems. This reduction clearly leads to a noticeable acceleration in computation, further highlighting the efficiency and scalability of our approach.

\begin{table}[H]
	\setlength{\belowcaptionskip}{-1pt}
	\caption{Numerical performance of the AS strategy applied to the generation of solution paths
	(20 different $\lambda$'s)  for the exclusive lasso logistic regression model on synthetic data sets.}
	\label{tab: AS-PPDNA-logistic-exclusive-lasso-Tab1}
	\renewcommand\arraystretch{1.2}
	\centering
	\begin{tabular}{|c|r|c @{\ $|$\ } c|c|c @{\ $|$\ } c|}  
		\hline
		& & \multicolumn{2}{c|}{Total time (hh:mm:ss)}  & \multicolumn{3}{c|}{Information of the AS} \\
		\hline
		$m$  & $n (= g \times p)$  & Warmstart & With AS  & S. Rnd & Avg. D. & Max. D. \\
		\hline
		\multirow{4}*{\tabincell{c}{$500$}}
		& $100,000$ & 00:01:03 & 00:00:08 & 31  & 306 & 3161 \\
		& $200,000$ & 00:01:47 & 00:00:10 & 34  & 385 & 4471 \\
		& $600,000$ & 00:04:41 & 00:00:13 & 30  & 443 & 7751 \\
		& $1,200,000$ & 00:10:16 & 00:00:19 &  30  & 497 & 10951 \\
		\hline
		\multirow{4}*{\tabincell{c}{$1000$}}
		& $100,000$ & 00:01:36 & 00:00:21 & 36  & 429 & 3161 \\
		& $200,000$ & 00:03:03 & 00:00:27 & 34  & 566 & 4471 \\
		& $600,000$ & 00:07:26 & 00:00:35 & 36  & 623 & 7751 \\
		& $1,200,000$ & 00:14:44 & 00:00:48 & 34  & 702 & 10951 \\
		\hline
		\multirow{4}*{\tabincell{c}{$2000$}}
		& $100,000$ & 00:02:58 & 00:01:16 & 36  & 499 & 3161 \\
		& $200,000$ & 00:05:12 & 00:01:36 & 37  & 661 & 4471 \\
		& $600,000$ &  00:13:30 & 00:02:11 & 38  & 899 & 7751 \\
		& $1,200,000$ & 00:26:45 & 00:02:40 & 38  &  1038 & 10951 \\
		\hline
	\end{tabular}
\end{table}

\subsection{Comparison between AS and other approaches}
\label{sec: allscreeningrules}
To further demonstrate the superior performance of the AS strategy, we will compare it with other popular approaches on the generation of solution paths for the Lasso linear regression models, which are implemented in high-quality publicly available packages. Specifically, we compare AS(+SSNAL) with the following approaches.
\begin{itemize}
	\item Matlab's CD: The coordinate descend method  \cite{friedman2010regularization}, with its Matlab implementation in built-in ``lasso" function in the Statistics and Machine Learning Toolbox.
	\item EDPP: The Enhanced Dual Projection onto Polytope (EDPP) rule implemented in the DPC Package\footnote{\url{http://dpc-screening.github.io/lasso.html}}, where the inner problems are solved by the function ``LeastR" in the solver SLEP \cite{liu2009slep}.
	\item FPC\_AS: The active-set based fixed point continuation method \cite{wen2010fast}.
\end{itemize}
Note that the above approaches are terminated by different stopping criteria: AS is terminated if $\eta_{\rm KKT}$ in \eqref{eq: eta_kkt} is no greater than a given tolerance $\varepsilon$; Matlab's CD terminates when successive estimates of $x$ differ in the $\ell_2$ norm by a relative amount less than {\tt RelTol}; EDPP is terminated if the relative successive estimates of $x$ is no greater than {\tt Tol}; and FPC\_AS terminates if the maximum norm of sub-gradient is smaller than {\tt gtol}. 

For fair comparison, we test each approach with two choices of tolerances, and then compare their running time as well as the relative objective function values against the benchmark. Here we set the benchmark as $f^*:=f(x^*) = \|Ax^*-b\|^2/2+\lambda \|x^*\|_1$, where $x^*$ is the solution obtained by the AS with a high accuracy ($\varepsilon=10^{-8}$). The relateive objective function value associated with an approximate solution $x$ is defined as $(f(x)-f^*)/f^*$. If this quantity is smaller, it means that the corresponding approximate solution is better.

The numerical comparison among the four approaches in generating solution paths for the Lasso model of size $(m,n)=(500,100000)$ with $\lambda_c$ decreasing from $1$ to $10^{-4}$ is shown in Table \ref{tab: random_lasso_path} and Figure \ref{fig: lasso_linear_path}. It can be seen from the results that, the AS strategy gives great performance on path generation, as it achieves higher accuracy solutions in less time comparing with the existing approaches. In addition, from the experimental results, we can see that, as the tolerance becomes more stringent, our AS strategy exhibits strong scalability in terms of running time.

\begin{table}[H]
	\caption{Running time of different approaches in generating solution paths for the Lasso model of size $(m,n)=(500,100000)$ with $\lambda_c$ decreasing from $1$ to $10^{-4}$.}
	\label{tab: random_lasso_path}
	\renewcommand\arraystretch{1.2}
	\centering
	\begin{tabular}{|c|c|c|c|c|c|c|c|}  
		\hline
		\multicolumn{2}{|c|}{AS}  & \multicolumn{2}{c|}{Matlab's CD}  & \multicolumn{2}{c|}{EDPP} & \multicolumn{2}{c|}{FPC\_AS} \\
		\hline
		$\varepsilon$ & Time & {\tt RelTol} & Time & {\tt Tol} & Time & {\tt gtol} & Time\\
		\hline 
		1e-6 & 00:00:09 & 1e-3 & 00:00:09 & 1e-4 & 00:01:12 & 1e-6 & 00:01:06\\
		\hline 
		1e-7 & 00:00:11 & 1e-4 & 00:00:27 & 1e-5 & 00:06:25 & 1e-7 & 00:04:28\\
		\hline 
	\end{tabular}
\end{table}

\vspace{-0.8cm}
\begin{figure}[H]
	\includegraphics[width = 1.08\columnwidth]{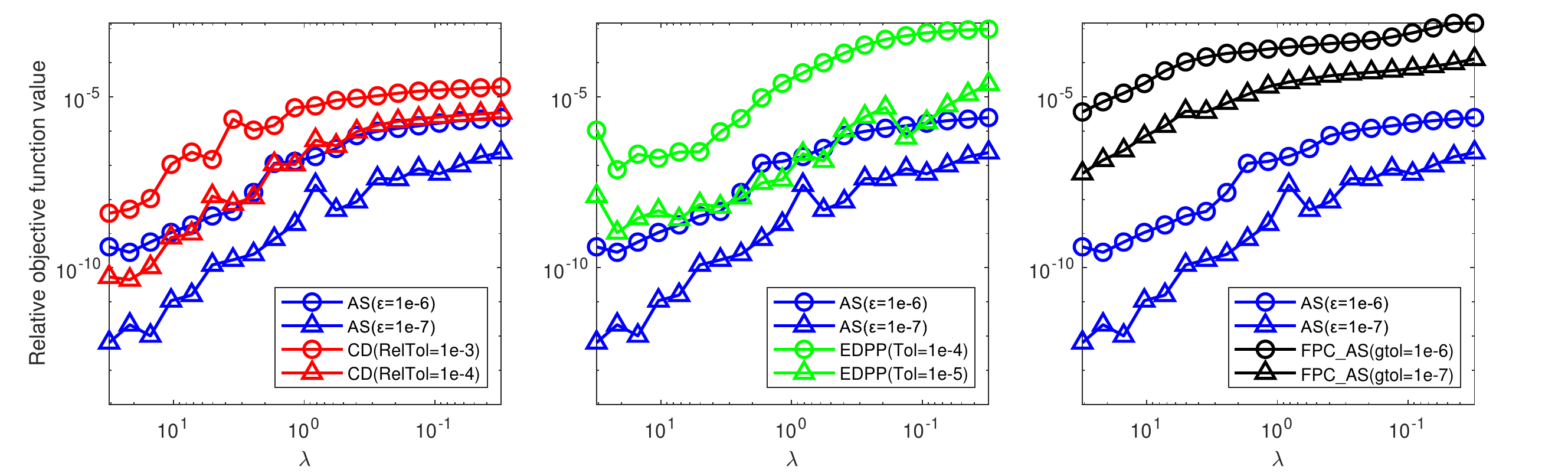}
	\setlength{\abovecaptionskip}{-10pt}
	\caption{Comparison among different approaches in generating solution paths for the Lasso model of size $(m,n)=(500,100000)$ with $\lambda_c$ decreasing from $1$ to $10^{-4}$.}
	\label{fig: lasso_linear_path}
\end{figure}

\subsection{Flexibility of the AS strategy}
\label{sec: algo_com}

In this subsection, we combine our AS strategy with different algorithms for solving the reduced problems to test its generality and flexibility. The path generation for the Lasso model is taken as an example. 

We incorporate the AS strategy with the SSNAL method, the alternating direction method of multipliers (ADMM) \cite{eckstein1992douglas,glowinski1975approximation}, the accelerated proximal gradient (APG) method \cite{beck2009fast}, and the coderivative-based generalized Newton method (GRNM) \cite{khanh2024globally}, separately. We stop the tested algorithm for solving each reduced problem if $\eta_{\rm KKT}\leq 10^{-6}$, the computation time exceeds 30 minutes, or the pre-set maximum number of iterations (100 for SSNAL, GRNM, and 20000 for ADMM, APG) is reached.

The numerical performance of the AS strategy combined with SSNAL, ADMM, APG and GRNM for solving the Lasso model of problem size $(m,n)=(500,100000)$ with $\lambda_c$ decreasing from $1$ to $10^{-4}$ is shown in Figure \ref{fig: lasso_linear_alg}. For illustration purpose, we also present the running time of each algorithm together with the warm-start techique for the path generation. As one can see from the figure, the AS strategy is highly efficient for generating solution paths of the Lasso models. Moreover, it can be incorporated with different algorithms for solving the reduced problems, and can consistently accelerate the computation. For example, when generating the solution path using the APG algorithm together with the warm-start technique, it takes a few hundred seconds for each $\lambda$, while with the AS technique, it takes less than one second. Additionally, the GRNM together with the warm-start technique is not able to solve the problems to the required accuracy within the specified time, while with the AS strategy, it is able to solve the problem for each $\lambda$ within tens of seconds. The improvement comes from the fact that GRNM needs to solve an $n\times n$ linear system for the Newton step, which is costly for high-dimensional problems. Moreover, since $n\gg m$, $A^TA$ becomes singular, leading to ill-conditioned (damped/regularized) Newton systems. In contrast, the AS strategy solves reduced, lower-dimensional problems, greatly speeding up GRNM and making it practical for high-dimensional tasks.

\vspace{-0.4cm}
\begin{figure}[H]
	\hspace{-0.3cm}
	\includegraphics[width = 1.15\columnwidth]{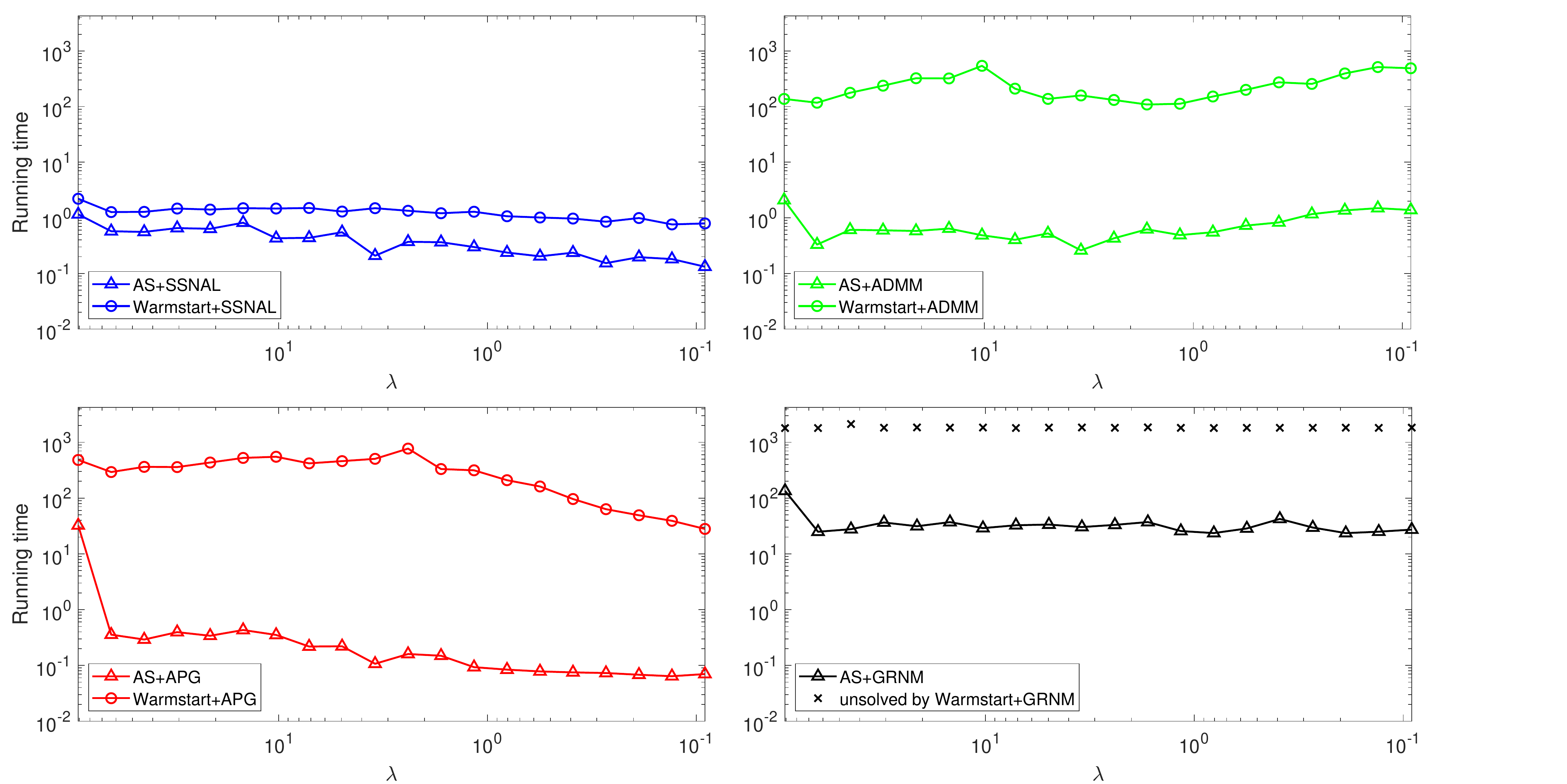}
	\setlength{\abovecaptionskip}{-5pt}
	\setlength{\belowcaptionskip}{-10pt}
	\caption{The combination of the AS strategy with different algorithms for solving the reduced problems during the path generation for the Lasso model of size $(m,n)=(500,100000)$. The maker ``{\tiny $^\times$}'' in the plot means that the problem is not solved to the required accuracy.}
	\label{fig: lasso_linear_alg}
\end{figure}

\subsection{Performance of the AS strategy on real data sets}
\label{sec:real}
In this subsection, we present the numerical performance of the AS strategy on the real data sets for the aforementioned four linear regression models. We test the AS strategy for the Lasso model and the SLOPE model on eight UCI data sets, where the group information is not required. In addition, we test the AS strategy for the sparse group lasso model and the exclusive lasso model on one climate data set, where the group information is available.

\begin{table}[H]
	\caption{Details of the UCI test data. $\lambda_{\max}(\cdot)$ denotes the largest eigenvalue.}
	\label{tab:uci_datasets}
	\renewcommand\arraystretch{1.2}
	\centering
	\begin{tabular}{|c|c|c|} \hline
		Name & $(m, n)$ &  $\lambda_{\max}(AA^T)$\\ \hline
		abalone7 & $(4177, 6435)$ & 5.21e+5\\ \hline
		bcTCGA & $(536, 17322)$ & 1.13e+7 \\ \hline
		bodyfat7 & $(252, 116280)$ & 5.29e+4\\ \hline
		housing7 & $(506,77520)$ & 3.28e+5 \\ \hline
		mpg7 & $(392,3432)$ & 1.28e+4 \\ \hline
		pyrim5 & $(74,201376)$ & 1.22e+6 \\ \hline
		space-ga9 & $(3107, 5005)$ & 4.01e+3 \\ \hline
		triazines4 & $(186, 635376)$ & 2.07e+7 \\ \hline
	\end{tabular}
\end{table}

\paragraph{Test instances from the UCI data repository.} We test eight data sets from the UCI data repository \cite{chang2011libsvm}. The details are summarized in Table \ref{tab:uci_datasets}.

Following the settings in \cite{li2018highly}, we expand the original features by using polynomial basis functions over those features for the data sets pyrim, triazines, abalone, bodyfat, housing, mpg, and space-ga. For instance, the digit $5$ in the name pyrim5 means that an order 5 polynomial is used to generate the basis functions.

\paragraph{NCEP/NCAR reanalysis 1 dataset.} The data set \cite{kalnay1996ncep} contains the monthly means of climate data measurements spread across the globe in a grid of $2.5^{o} \times 2.5^{o}$ resolutions (longitude and latitude 144×73) from 1948/1/1 to 2018/5/31. Each grid point (location) constitutes a group of 7 predictive variables (Air Temperature, Precipitable Water, Relative Humidity, Pressure, Sea Level Pressure, Horizontal Wind Speed and Vertical Wind Speed). Such data sets have two natural group structures: (i) groups over locations: $144 \times 73$ groups, where each group is of length 7; (ii) groups over features: 7 groups, where each group is of length 10512. For both cases, the corresponding data matrix A is of dimension $845 \times 73584$.

\subsubsection{Numerical results on Lasso linear regression model}
In this subsection, we present the numerical results of the AS strategy for the Lasso linear regression model on the test instances from the UCI data repository.

\begin{table}[H]
	\caption{Running time of different approaches in generating solution paths for the Lasso model on mpg7 data set with $\lambda_c$ decreasing from $1$ to $10^{-4}$.}
	\label{tab: mpg_lasso_path}
	\renewcommand\arraystretch{1.2}
	\centering
	\begin{tabular}{|c|c|c|c|c|c|c|c|}  
		\hline
		\multicolumn{2}{|c|}{AS}  & \multicolumn{2}{c|}{Matlab's CD}  & \multicolumn{2}{c|}{EDPP} & \multicolumn{2}{c|}{FPC\_AS} \\
		\hline
		$\varepsilon$ & Time & {\tt RelTol} & Time & {\tt Tol} & Time & {\tt gtol} & Time\\
		\hline 
		1e-6 & 00:00:06 & 1e-3 & 00:00:12 & 1e-4 & 00:00:04 & 1e-5 & 00:00:11\\
		\hline 
		1e-7 & 00:00:06 & 1e-4 & 00:01:50 & 1e-5 & 00:00:22 & 1e-6 & 00:00:46\\
		\hline 
	\end{tabular}
\end{table}

\begin{figure}[H]
	\includegraphics[width = \columnwidth]{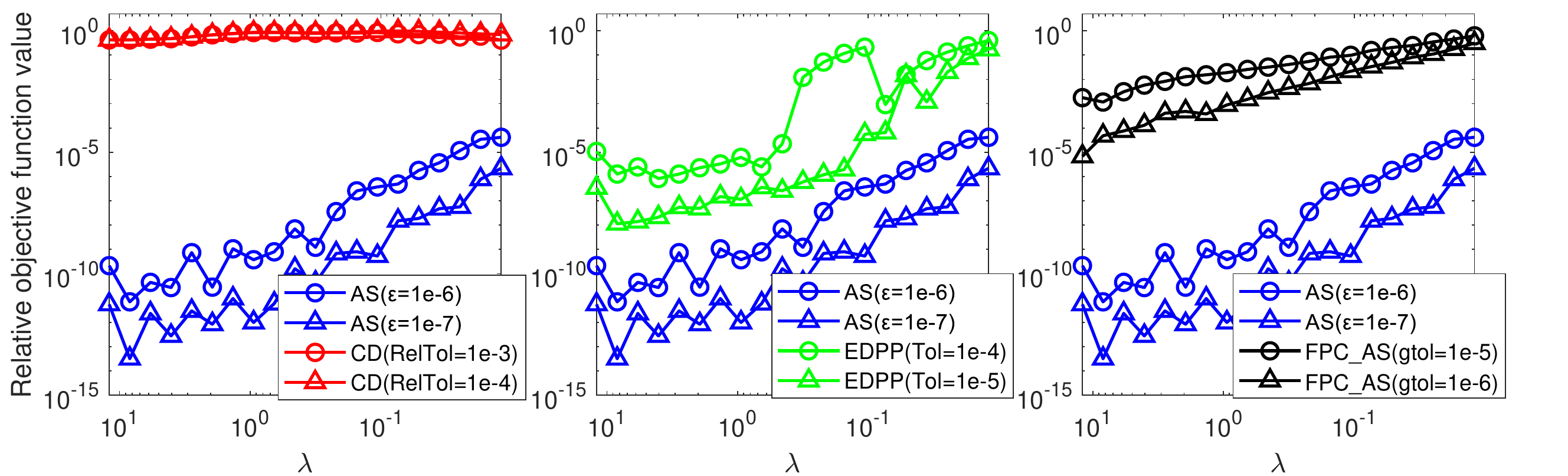}
	\caption{Comparison among different approaches in generating solution paths for the Lasso model on mpg7 data set with $\lambda_c$ decreasing from $1$ to $10^{-4}$.}
	\label{fig: mpg_path}
\end{figure}

First, we compare the AS strategy with Matlab's CD, EDPP, and FPC\_AS on the path generation for the small data set mpg7. The running time of each approach with different choices of tolerances is shown in Table \ref{tab: mpg_lasso_path}, and the relative objective function values along the solution path obtained by each approach are shown in Figure \ref{fig: mpg_path}. It can be seen that the AS strategy outperforms other approaches by achieving more accurate solutions in less time. This is, it stands out in both speed and solution quality.

Next, we test the performance of the AS strategy combined with different algorithms for solving the reduced problems when solving the Lasso model on the small mpg7 data set. The results are shown in Figure \ref{fig: mpg_alg}. We can see that the AS strategy can consistently improve the performance of SSNAL, ADMM, APG and GRNM, which again demonstrates its effectiveness and flexibility.

\begin{figure}[H]
	\hspace{-0.3cm}
	\includegraphics[width = 1.15\columnwidth]{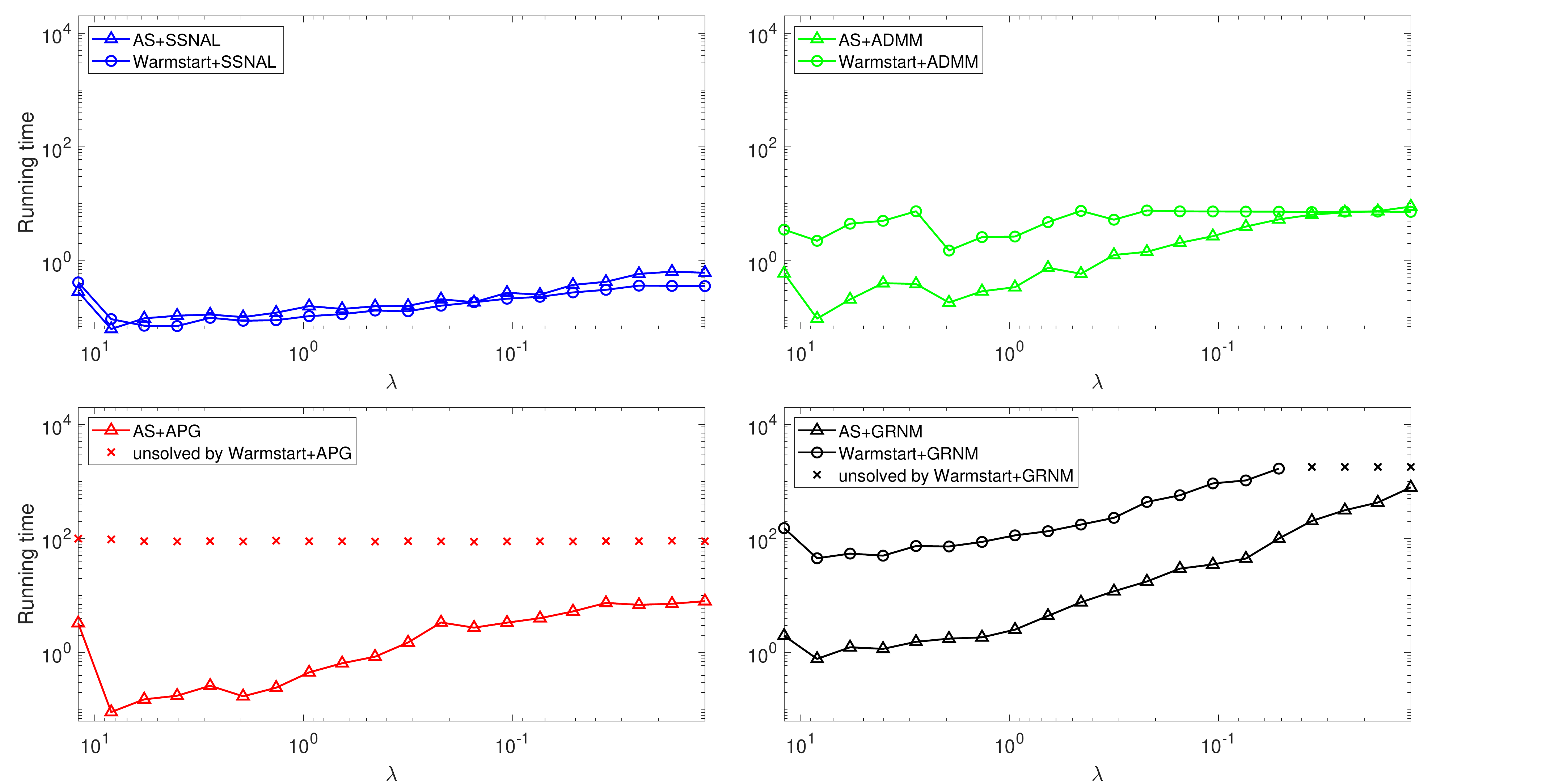}
	\setlength{\abovecaptionskip}{-5pt}
	\caption{The combination of the AS strategy with different algorithms for solving the reduced problems for the Lasso model on mpg7 data set.}
	\label{fig: mpg_alg}
\end{figure}

To better illustrate the generality and flexibility of the proposed AS strategy, in addition to the experiments on the small mpg7 data set, we conduct more experiments on other larger UCI data sets. For the algorithms to solve the corresponding models and their reduced problems, we take SSNAL and ADMM as two examples for illustration. We follow the same setting as described in Section \ref{sec: lasso-linear-synthetic} to choose the values for the parameter $\lambda$. The results are summarized in Table \ref{tab: AS-PPDNA-lasso-uci}. 

As one can see from Table \ref{tab: AS-PPDNA-lasso-uci}, the SSNAL algorithm together with the warm-start technique already gives great performance, which takes less than 1 minute to generate a solution path on each of the first seven data sets and takes around 8 minutes to generate a solution path on the last data set. Surprisingly, the SSNAL algorithm together with our proposed AS strategy can further speed up the path generation by a large margin, which can generate solution paths on all data sets within 45 seconds. In addition, the AS strategy also shows excellent performance when the reduced problems are solved by the ADMM algorithm. For example, for bodyfat7 data set, the ADMM algorithm with the warm-start technique takes more than 2 hours to generate a solution path, while it only takes around 20 seconds with the AS strategy. Moreover, for the pyrim5 data set, the AS strategy allows the ADMM algorithm to generate a solution path with the required accuracy within 1 minute, while it takes 3 hours with the warm-start technique to generate a solution path with low accuracy.

\begin{table}[H]
	\caption{Numerical performance of the AS strategy applied to the generation of solution paths for Lasso linear regression models on UCI instances. In the table, $\eta^*_{\rm KKT}$ denotes the worst relative KKT residual of the problems solved on the solution path, Org.D. denotes the original problem dimension, and the values in bold mean that the required accuracy not achieved.}
	\label{tab: AS-PPDNA-lasso-uci}
	\renewcommand\arraystretch{1.2}
	\centering
	\begin{tabular}{|@{\hspace{2pt}}c@{\hspace{2pt}}|@{\hspace{3pt}}c@{\hspace{3pt}}|c @{\ $|$\ } c @{\hspace{2pt}}|c@{\ $|$\ } c|@{\hspace{2pt}} c  @{\ $|$\ } c  @{\ $|$\ } c @{\hspace{2pt}}|}
		\hline
		&& \multicolumn{2}{c|}{Warmstart}  & \multicolumn{5}{c|}{With AS}  \\
		\hline
		Data(Org.D.) & Alg. & Time  & $\eta^*_{\rm KKT}$  &  Time & $\eta^*_{\rm KKT}$ & S.Rnd & Avg.D. & Max.D.\\ 
		\hline
		\multirow{2}*{\tabincell{c}{abalone7 \\ (6435)} }  & {\fontsize{7.3}{5}\selectfont SSNAL} & 00:00:32 & 8.02e-7 & 00:00:22 & 9.82e-7 & 41 & 291 & 1505 \\
		& {\fontsize{7.3}{5}\selectfont ADMM} & 01:16:54 & 1.00e-6 & 00:18:13 & 9.99e-7 & 41 & 293 & 1505\\
		\hline
		\multirow{2}*{\tabincell{c}{bcTCGA \\ (17322)} } & {\fontsize{7.3}{5}\selectfont SSNAL} & 00:00:11 & 9.82e-7  & 00:00:08 & 9.89e-7 & 38 & 547 & 1908 \\
		& {\fontsize{7.3}{5}\selectfont ADMM} & 00:14:24 & 1.00e-6   & 00:00:11 & 9.76e-7 & 39 & 547 & 1908\\
		\hline
		\multirow{2}*{\tabincell{c}{bodyfat7 \\ (116280)} } & {\fontsize{7.3}{5}\selectfont SSNAL} & 00:00:23 & 9.05e-7  & 00:00:03 & 9.87e-7 & 39 & 574 & 6390 \\
		& {\fontsize{7.3}{5}\selectfont ADMM} & 02:07:17 & {\fontsize{7}{9}\selectfont \textbf{8.44e-6}}  &  00:00:21 & 9.94e-7 & 38  &  578 & 6391\\
		\hline
		\multirow{2}*{\tabincell{c}{housing7 \\ (77520)} } & {\fontsize{7.3}{5}\selectfont SSNAL} & 00:00:37 & 9.56e-7  & 00:00:13 & 9.78e-7 & 57 & 1129 & 12634\\
		& {\fontsize{7.3}{5}\selectfont ADMM} & 02:30:10 & {\fontsize{7}{9}\selectfont \textbf{3.01e-5}}  & 00:02:14 & 9.98e-7 & 57 & 1129 & 12634\\
		\hline
		\multirow{2}*{\tabincell{c}{pyrim5 \\ (201376)} } & {\fontsize{7.3}{5}\selectfont SSNAL} & 00:00:57 & 9.75e-7 & 00:00:03 & 9.79e-7 & 41 & 492 & 4693  \\
		& {\fontsize{7.3}{5}\selectfont ADMM} & 03:00:55 & {\fontsize{7}{9}\selectfont \textbf{1.45e-4}} & 00:00:53 & 9.92e-7 & 44 & 491 & 4693 \\
		\hline
		\multirow{2}*{\tabincell{c}{space-ga9 \\ (5005)} } & {\fontsize{7.3}{5}\selectfont SSNAL} & 00:00:19 & 9.72e-7 & 00:00:12 & 9.93e-7 & 40 & 318 & 992  \\
		& {\fontsize{7.3}{5}\selectfont ADMM} &00:28:48 & 1.00e-6 & 00:08:48 & 9.99e-7 & 41 & 324 & 993\\
		\hline
		\multirow{2}*{\tabincell{c}{triazines4 \\ (635376)} } & {\fontsize{7.3}{5}\selectfont SSNAL} & 00:08:05 & 9.88e-7  & 00:00:44 & 9.90e-7 & 53 & 10187 & 55935  \\
		& {\fontsize{7.3}{5}\selectfont ADMM} & 10:00:09 & {\fontsize{7}{9}\selectfont \textbf{9.10e-3} }  & 02:46:28 & {\fontsize{7}{9}\selectfont \textbf{1.23e-5}} & 68 & 16179 & 55932  \\
		\hline
	\end{tabular}
\end{table}

\subsubsection{Numerical results on SLOPE linear regression model}
\label{sec: slope_real}
In this subsection, we present the numerical results of the AS strategy for the SLOPE linear regression model on the test instances from the UCI data repository. For the weights $\lambda_{(i)}$ in the SLOPE model, we follow the experiment settings in \cite{bogdan2015slope}. We test the numerical performance on generating a solution path for the SLOPE regression model with $\lambda =\lambda_c\|A^T b\|_{\infty}$, where $\lambda_c$ is taken from $10^{-1}$ to $10^{-4}$ with 20 equally divided grid points on the $\log_{10}$ scale. As for the algorithm to solve the corresponding models and their reduced problems, we take the state-of-the-art SSNAL method \cite{luo2019solving} and the popular first-order method ADMM for examples. It should be noted that, the AS strategy can be combined with any solver for solving the reduced problems. The results are summarized in Table \ref{tab: AS-PPDNA-SLOPE-uci}.

\begin{table}[H]
	\caption{Numerical performance of the AS strategy applied to the generation of solution paths for SLOPE linear regression models on UCI instances. }
	\label{tab: AS-PPDNA-SLOPE-uci}
	\renewcommand\arraystretch{1.2}
	\centering
	\begin{tabular}{|@{\hspace{2pt}}c@{\hspace{2pt}}|@{\hspace{3pt}}c@{\hspace{3pt}}|c @{\ $|$\ } c@{\hspace{2pt}}|c @{\ $|$\ } c|@{\hspace{2pt}}c  @{\ $|$\ } c  @{\ $|$\ } c @{\hspace{2pt}}|}
		\hline
		&& \multicolumn{2}{c|}{Warmstart}  & \multicolumn{5}{c|}{With AS}  \\
		\hline
		Data(Org.D.) & Alg. & Time  & $\eta^*_{\rm KKT}$  &  Time & $\eta^*_{\rm KKT}$ & S.Rnd & Avg.D. & Max.D.\\ 
		\hline
		\multirow{2}*{\tabincell{c}{abalone7 \\ (6435)} }  & {\fontsize{7.3}{5}\selectfont SSNAL} & 00:01:02 & 9.69e-7 & 00:00:31 & 4.12e-7 & 19 & 1010 & 4565 \\
		& {\fontsize{7.3}{5}\selectfont ADMM} & 00:42:14 & {\fontsize{7}{9}\selectfont \textbf{6.35e-6}} & 00:09:34 & 9.58e-7 & 22 & 1049 & 4565\\
		\hline
		\multirow{2}*{\tabincell{c}{bcTCGA \\ (17322)} } & {\fontsize{7.3}{5}\selectfont SSNAL} & 00:00:33 & 6.09e-7  & 00:00:23 & 3.61e-7 & 24 & 571 & 899 \\
		& {\fontsize{7.3}{5}\selectfont ADMM} & 00:02:44  & 9.99e-7  & 00:00:05 & 9.29e-7 & 24 & 571 & 899\\
		\hline
		\multirow{2}*{\tabincell{c}{bodyfat7 \\ (116280)} } & {\fontsize{7.3}{5}\selectfont SSNAL} & 00:01:24 & 8.81e-7 & 00:01:11 & 7.93e-7 & 14 & 19275 & 65275 \\
		& {\fontsize{7.3}{5}\selectfont ADMM} & 00:35:35 & {\fontsize{7}{9}\selectfont  \textbf{1.16e-5}} &  00:02:54 & 9.06e-7 & 15 & 16871 & 41866\\
		\hline
		\multirow{2}*{\tabincell{c}{housing7 \\ (77520)} } & {\fontsize{7.3}{5}\selectfont SSNAL} & 00:01:09 & 8.52e-7 & 00:00:23 & 5.84e-7 & 16 & 3036 & 11191\\
		& {\fontsize{7.3}{5}\selectfont ADMM} & 00:18:29 & {\fontsize{7}{9}\selectfont  \textbf{1.03e-6}} & 00:00:23 & 5.86e-7 & 16 & 2974 & 9299\\
		\hline
		\multirow{2}*{\tabincell{c}{mpg7 \\ (3432)} } & {\fontsize{7.3}{5}\selectfont SSNAL} & 00:00:06 & 6.70e-7 & 00:00:04 & 6.15e-7 & 21 & 321 & 982  \\
		& {\fontsize{7.3}{5}\selectfont ADMM} & 00:00:25 & 1.00e-6 &  00:00:04 & 9.32e-7 & 22 & 310 & 544 \\
		\hline
		\multirow{2}*{\tabincell{c}{pyrim5 \\ (201376)} } & {\fontsize{7.3}{5}\selectfont SSNAL} & 00:01:27 & 9.87e-7 & 00:00:10 & 4.04e-7 & 16 & 1948 & 12191 \\
		& {\fontsize{7.3}{5}\selectfont ADMM} & 01:23:55 & {\fontsize{7}{9}\selectfont  \textbf{9.92e-4}} & 00:00:12 & 7.49e-7 & 16 & 1948 & 12191 \\
		\hline
		\multirow{2}*{\tabincell{c}{space-ga9 \\ (5005)} } & {\fontsize{7.3}{5}\selectfont SSNAL} & 00:00:31 & 8.54e-7 & 00:00:10 & 5.56e-7 & 18 & 246 & 1427  \\
		& {\fontsize{7.3}{5}\selectfont ADMM} &00:07:30 & 9.90e-7 & 00:01:56 & 8.50e-7 & 21 & 205 & 616\\
		\hline
		\multirow{2}*{\tabincell{c}{triazines4 \\ (635376)} } & {\fontsize{7.3}{5}\selectfont SSNAL} & 00:16:39 & 8.26e-7 & 00:00:42 & 4.64e-7 & 16 & 3779 &  29465 \\
		& {\fontsize{7.3}{5}\selectfont ADMM} & 08:52:56 & {\fontsize{7}{9}\selectfont  \textbf{1.22e-4}} & 00:01:19 & 8.05e-7 & 16 & 3778 & 29465  \\
		\hline
	\end{tabular}
\end{table}

From Table \ref{tab: AS-PPDNA-SLOPE-uci}, we again observe the extremely good performance of the proposed AS strategy when combined with the SSNAL algorithm or the ADMM algorithm. For example, for the triazines4 data set, the SSNAL algorithm together with the warm-start technique takes more than 16 minutes to generate a solution path, while it only takes 42 seconds with the AS strategy. For the pyrim5 data set, the ADMM algorithm together with the warm-start technique takes more than one hour and 23 minutes to generate a solution path, while with the help of the AS strategy, it only takes 12 seconds for the path generation. From these experimental results, we can see that our proposed AS strategy can accelerate optimization algorithms for solving large-scale sparse optimization problems with intrinsic structured sparsity.

\subsubsection{Numerical results on exclusive lasso linear regression model}
We test the performance of the AS strategy for the exclusive lasso linear regression model on the NCEP/NCAR reanalysis 1 dataset. Since the exclusive lasso regularizer induces intra-group feature selections, we choose the group structure over features. In other words, there will be seven groups, and each group includes feature values of 10512 different locations. For the choice of the values for the parameter $\lambda$ and the weight vector $w$, we follow the same setting as described in Section \ref{sec: exclusive-lasso-linear-synthetic}. To solve the reduced problems, we apply the PPDNA algorithm and the ADMM algorithm for illustration. The results are summarized in Table \ref{tab: AS-PPDNA-exclusivelasso-climate}. We can observe that the AS strategy can accelerate the PPDNA algorithm by around 10 times and can accelerate the ADMM algorithm by more than 228 times. In addition, it also gives great performance in reducing the dimensions of the problems to be solved.

\begin{table}[H]
	\caption{Numerical performance of the AS strategy applied to the generation of solution paths for the exclusive lasso linear regression models on the NCEP/NCAR reanalysis 1 dataset.}
	\label{tab: AS-PPDNA-exclusivelasso-climate}
	\renewcommand\arraystretch{1.2}
	\centering
	\begin{tabular}{|@{\hspace{2pt}}c@{\hspace{1pt}}|@{\hspace{3pt}}c@{\hspace{3pt}}|c @{\ $|$\ } c@{\hspace{2pt}}|c @{\ $|$\ } c|@{\hspace{2pt}}c  @{\ $|$\ } c  @{\ $|$\ } c @{\hspace{2pt}}|}
		\hline
		& & \multicolumn{2}{c|}{Warmstart}  & \multicolumn{5}{c|}{With AS}  \\
		\hline
		Data(Org.D.) & Alg. & Time  & $\eta^*_{\rm KKT}$  &  Time & $\eta^*_{\rm KKT}$ & S.Rnd & Avg.D. & Max.D.\\ 
		\hline 
		\multirow{2}*{\tabincell{c}{NCEP/NCAR \\ (73584)} }  & {\fontsize{7.3}{5}\selectfont PPDNA} & 00:01:18 & 9.85e-7  & 00:00:08 & 9.25e-7 & 14 & 178 & 2810 \\
		& {\fontsize{7.3}{5}\selectfont ADMM} & 03:10:27 & {\fontsize{7}{9}\selectfont  \textbf{3.96e-4}} & 00:00:50 & 9.39e-7 & 14 & 179 & 2810\\
		\hline
	\end{tabular}
\end{table}

\subsubsection{Numerical results on sparse group lasso linear regression model}
\label{sec: sglasso_real}
We test the performance of the AS strategy for the sparse group lasso linear regression model on the NCEP/NCAR reanalysis 1 dataset. We follow the numerical experiment settings in \cite{zhang2018efficient}, and we choose the group structure over locations. In other words, there will be 10512 groups, and each group includes 7 features. Following the settings in \cite{zhang2018efficient}, we test the numerical performance of the AS strategy for generating a solution path for the sparse group lasso linear regression model with $w_l = \sqrt{|G_l|}$ for each $l=1,\cdots,g$. For the parameters $\lambda_1$ and $\lambda_2$ in the sparse group lasso regularizer, we take $\lambda_1 =  \lambda_c\|A^T b\|_{\infty}$, $\lambda_2 = \lambda_1/w_{\max}$ with $w_{\max} = \max \{w_1, \dots, w_g\}$. Here, $\lambda_c$ is taken from $10^{-1}$ to $10^{-3}$ with 20 equally divided grid points on the $\log_{10}$ scale. As for solving the reduced problems in the AS procedure, we employ the state-of-the-art SSNAL method \cite{zhang2018efficient}. The comparison of the efficiency and robustness between the SSNAL method and other popular algorithms for solving sparse group lasso models is presented in \cite{zhang2018efficient}. In addition, the popular first-order method ADMM algorithm is also employed.

The results are summarized in Table \ref{tab: AS-SSNAL-grouplasso-climate}. It can be seen from the table that, the AS strategy can accelerate the SSNAL algorithm by around 54 times and accelerate the ADMM algorithm by around 187 times.

\begin{table}[H]
	\caption{Numerical performance of the AS strategy applied to the generation of solution paths for the sparse group lasso linear regression models on the NCEP/NCAR reanalysis 1 dataset.}
	\label{tab: AS-SSNAL-grouplasso-climate}
	\renewcommand\arraystretch{1.2}
	\centering
	\begin{tabular}{|@{\hspace{2pt}}c@{\hspace{1pt}}|@{\hspace{3pt}}c@{\hspace{3pt}}|c @{\ $|$\ } c@{\hspace{2pt}}|c@{\ $|$\ } c|@{\hspace{2pt}}c  @{\ $|$\ } c  @{\ $|$\ } c @{\hspace{2pt}}|}
		\hline
		& & \multicolumn{2}{c|}{Warmstart}  & \multicolumn{5}{c|}{With AS}  \\
		\hline
		Data(Org.D.) & Alg. & Time  & $\eta^*_{\rm KKT}$  &  Time & $\eta^*_{\rm KKT}$ & S.Rnd & Avg.D. & Max.D.\\ 
		\hline 
		\multirow{2}*{\tabincell{c}{NCEP/NCAR \\ (73584)} }  & {\fontsize{7.3}{5}\selectfont PPDNA} & 00:24:33 & 4.81e-7 & 00:00:27 & 9.81e-7 & 22 & 146 & 2711 \\
		& {\fontsize{7.3}{5}\selectfont ADMM} & 02:48:46 & {\fontsize{7}{9}\selectfont \textbf{4.40e-3}} & 00:00:54 & 8.83e-7 & 24 & 145 & 2711\\
		\hline
	\end{tabular}
\end{table}

\section{Conclusion}
\label{sec:conclusion}
In this paper, we design an adaptive sieving strategy for solving general sparse optimization models and further generating solution paths for a given sequence of parameters. For each reduced problem involved in the AS strategy, we allow it to be solved inexactly. The finite termination property of the AS strategy has also been established. Extensive numerical experiments have been conducted to demonstrate the effectiveness and flexibility of the AS strategy to solve large-scale machine learning models.

\section*{Funding}
Yancheng Yuan is supported in part by The Hong Kong Polytechnic University under Grants P0038284/P0045485, and the Research Center for Intelligent Operations Research. Meixia Lin is supported by the Ministry of Education, Singapore, under its Academic Research Fund Tier 2 grant call (MOE-T2EP20123-0013). Defeng Sun is supported in part by the Hong Kong Research Grant Council under Grant 15304721. Kim-Chuan Toh is supported by the Ministry of Education, Singapore, under its Academic Research Fund Tier 3 grant call (MOE-2019-T3-1-010).

\begin{appendix}
{\normalsize
\section{More numerical experiments on synthetic data}
In this section, we present more experiments to demonstrate the performance of the AS strategy to generate solution paths for sparse optimization problems. 

\subsection{Numerical results on sparse group lasso linear regression problems}
\label{sec: sparse-group-lasso-linear-regression-synthetic}
In this subsection, we present the numerical results on the sparse group lasso linear regression model. We apply the state-of-the-art SSNAL method \cite{zhang2018efficient} to solve the sparse group lasso model and its corresponding reduced problems generated by the AS strategy during the path generation. For the choice of the values for the parameters $\lambda_1$, $\lambda_2$ and $w$, we follow the same setting as described in Section \ref{sec: sglasso_real}, and $\lambda_c$ is taken from $10^{-1}$ to $10^{-4}$ with 20 equally divided grid points on the $\log_{10}$ scale.

\begin{table}[H]
	\caption{Numerical performance of the AS strategy applied to the generation of solution paths for the sparse group lasso linear regression model on synthetic data sets with sparsity level $0.1\%$.}
	\label{tab: AS-PPDNA-sgrouplasso-Tab1}
	\renewcommand\arraystretch{1.2}
	\centering
	\begin{tabular}{|c|r|c @{\ $|$\ } c|c|c @{\ $|$\ } c|}  
		\hline
		& & \multicolumn{2}{c|}{Total time (hh:mm:ss)}  & \multicolumn{3}{c|}{Information of the AS} \\
		\hline
		$m$  & $n (= g \times p)$  & Warmstart & With AS  & S. Rnd & Avg. D. & Max. D. \\
		\hline
		\multirow{4}*{\tabincell{c}{$500$}}
		& $100,000$ & 00:03:00 & 00:00:35 & 17  & 1743 & 3588 \\
		& $200,000$ & 00:04:58 & 00:00:38 & 16  & 2234 & 4829 \\
		& $600,000$ & 00:14:11 & 00:00:55 & 21  & 2764 & 7839 \\
		& $1,200,000$ & 00:27:34 & 00:01:00 & 20  & 3215 & 10998 \\
		\hline
		\multirow{4}*{\tabincell{c}{$1000$}}
		& $100,000$ & 00:05:47 & 00:02:36 & 22  & 2000 & 4259\\
		& $200,000$ & 00:10:28 & 00:02:16 & 19  & 2827 & 6295\\
		& $600,000$ & 00:31:02 & 00:02:16 & 17  & 3570 & 8311 \\
		& $1,200,000$ & 00:52:40 & 00:02:21 & 16  & 4286 & 11252 \\
		\hline
		\multirow{4}*{\tabincell{c}{$2000$}}
		& $100,000$ & 00:15:46 & 00:06:19 & 34  & 2062 & 3464 \\
		& $200,000$ & 00:23:18 & 00:06:46 & 33  & 3006 & 6321 \\
		& $600,000$ & 00:58:03 & 00:05:31 & 25  & 4742 & 11063 \\
		& $1,200,000$ & 01:55:24 & 00:07:07 & 26  & 5327 & 12557 \\
		\hline
	\end{tabular}
\end{table}

We summarize the numerical results of the comparison of the AS strategy and the warm-start technique for generating solution paths of sparse group lasso linear regression problems in Table \ref{tab: AS-PPDNA-sgrouplasso-Tab1}. As we can see from the table, the AS strategy significantly accelerates the path generation of the sparse group lasso linear regression models. For example, on the synthetic data set of size $(m,n)=(2000,1200000)$, generating a solution path of the sparse group lasso problem with the warm-start technique takes around $2$ hours, while with our proposed AS strategy, it only takes around $7$ minutes. The superior performance of the AS strategy can also been seen from the much smaller sizes of the reduced problems compared with the original problem sizes.

\subsection{Numerical results on SLOPE linear regression problems}
\label{sec: SLOPE-linear-regression-synthetic}
Here, we present the results on SLOPE linear regression model. We will apply the state-of-the-art SSNAL method \cite{luo2019solving} to solve the SLOPE regression model and its corresponding reduced problems generated by the AS strategy. The comparison of the efficiency between the SSNAL method and other popular algorithms for solving SLOPE models can be found in \cite{luo2019solving}. The choice of the values for the parameter $\lambda$ follows from the same setting as described in Section \ref{sec: slope_real}.

The results are presented in Table \ref{tab: AS-PPDNA-SLOPE-Tab1}. We again observe the superior performance of the AS strategy on solving sparse optimization models. Specifically, when we apply the warm-start technique to generate solution paths for the SLOPE linear regression model with $(m,n)=(2000,1200000)$, it runs out of memory, while with our AS strategy, it takes less than 12 minutes. Moreover, our AS strategy gives great performance in reducing the problem sizes during the path generation.

\begin{table}[H]
	\caption{Numerical performance of the AS strategy applied to the generation of solution paths for the SLOPE linear regression model on synthetic data sets with sparsity level $0.1\%$.}
	\label{tab: AS-PPDNA-SLOPE-Tab1}
	\renewcommand\arraystretch{1.2}
	\centering
	\begin{tabular}{|c|r|c @{\ $|$\ } c|c|c @{\ $|$\ } c|}  
		\hline
		& & \multicolumn{2}{c|}{Total time (hh:mm:ss)}  & \multicolumn{3}{c|}{Information of the AS} \\
		\hline
		$m$  & $n (= g \times p)$  & Warmstart & With AS  & S. Rnd & Avg. D. & Max. D. \\
		\hline
		\multirow{4}*{\tabincell{c}{$500$}}
		& $100,000$ & 00:01:38 & 00:00:38 & 22  & 1484 & 3122 \\
		& $200,000$ & 00:02:42 & 00:00:36 & 19  & 1962 & 4474 \\
		& $600,000$ & 00:06:30 & 00:00:37 & 17  & 2314 & 3922 \\
		& $1,200,000$ & 00:13:31 & 00:00:44 &  14  & 3140 & 10954 \\
		\hline
		\multirow{4}*{\tabincell{c}{$1000$}}
		& $100,000$ & 00:04:30 & 00:02:33 & 21  & 2007 & 3000 \\
		& $200,000$ & 00:05:46 & 00:02:09 & 19  & 2981 & 4701 \\
		& $600,000$ & 00:14:44 & 00:02:44 & 21  & 4476 & 7988 \\
		& $1,200,000$ & 00:27:17 & 00:02:38 & 17  & 5050 & 11151 \\
		\hline
		\multirow{4}*{\tabincell{c}{$2000$}}
		& $100,000$ & 00:10:24 & 00:05:39 & 21  & 2514 & 4914 \\
		& $200,000$ & 00:17:31 & 00:09:38 & 21  & 3571 & 5553 \\
		& $600,000$ & 00:31:42 & 00:10:29 & 22  & 7038 & 8619 \\
		& $1,200,000$ & out-of-memory & 00:11:26 & 23  & 8322 & 11234 \\
		\hline
	\end{tabular}
\end{table}

\subsection{Numerical results on data sets with other sparsity level}
\label{sec: other_sparsity}
As mentioned in Section \ref{sec:linear_synthetic}, in the data generation procedure, we set each group of the ground-truth $x^*$ to contain $r$ nonzero elements. We have tested the case when $r = \lfloor 0.1\% \times p \rfloor$ in Sections \ref{sec:linear_synthetic}-\ref{sec: algo_com}, \ref{sec: sparse-group-lasso-linear-regression-synthetic} and \ref{sec: SLOPE-linear-regression-synthetic}. In this subsection, we present more experimental results for the case when $r = \lfloor 1\% \times p \rfloor$.

\begin{table}[H]
	\centering
	\caption{Numerical performance of the AS strategy applied to the generation of solution paths for the Lasso linear regression model on synthetic data sets with sparsity level at $1\%$. }
	\label{tab: AS-PPDNA-lasso-Tab3}
	\renewcommand\arraystretch{1.2}
	\begin{tabular}{|c|r|c @{\ $|$\ } c|c|c @{\ $|$\ } c|}  
		\hline
		& & \multicolumn{2}{c|}{Total time (hh:mm:ss)}  & \multicolumn{3}{c|}{Information of the AS} \\
		\hline
		$m$  & $n (= g \times p)$  & Warmstart & With AS  & S. Rnd & Avg. D. & Max. D. \\
		\hline
		\multirow{4}*{\tabincell{c}{$500$}}
		& $100,000$ & 00:00:53 & 00:00:07  & 31  & 682 & 3310 \\
		& $200,000$ & 00:01:14 & 00:00:07  & 26  & 765 & 4503 \\
		& $600,000$ & 00:04:49 & 00:00:09  & 25  & 830 & 7760 \\
		& $1,200,000$ & 00:08:46 & 00:00:12  & 23  & 997 & 10956 \\
		\hline
		\multirow{4}*{\tabincell{c}{$1000$}}
		& $100,000$ & 00:01:17 & 00:00:25  & 29  & 1304 & 4958 \\
		& $200,000$ & 00:02:28 & 00:00:27  & 29  & 1345 & 5180 \\
		& $600,000$ & 00:07:03 & 00:00:34  & 31  & 1554 & 7898 \\
		& $1,200,000$ & 00:12:55 & 00:00:38  & 27  & 1644 & 11035 \\
		\hline
		\multirow{4}*{\tabincell{c}{$2000$}}
		& $100,000$ & 00:03:20 & 00:02:14  & 33  & 2372 & 11074 \\
		&  $200,000$ & 00:05:04 & 00:02:17  & 34  & 2574 & 11172 \\
		& $600,000$ & 00:13:27 & 00:02:22  &32 & 2557 & 9357 \\
		&  $1,200,000$ & 00:25:30 & 00:02:39  & 31  & 2773 & 11757 \\
		\hline
	\end{tabular}
\end{table}

\begin{table}[H]
	\caption{Numerical performance of the AS strategy applied to the generation of solution paths for the exclusive lasso linear regression model on synthetic data sets with sparsity level at $1\%$.}
	\label{tab: AS-PPDNA-exclusivelasso-Tab3}
	\renewcommand\arraystretch{1.2}
	\centering
	\begin{tabular}{|c|r|c @{\ $|$\ } c|c|c @{\ $|$\ } c|}  
		\hline
		& & \multicolumn{2}{c|}{Total time (hh:mm:ss)}  & \multicolumn{3}{c|}{Information of the AS} \\
		\hline
		$m$  & $n (= g \times p)$  & Warmstart & With AS  & S. Rnd & Avg. D. & Max. D. \\
		\hline
		\multirow{4}*{\tabincell{c}{$500$}}
		& $100,000$ & 00:00:34 & 00:00:05 & 29  & 199 & 3161\\
		& $200,000$ & 00:01:09 & 00:00:05 & 30  & 206 & 4471 \\
		& $600,000$ & 00:03:06 & 00:00:08 & 24 & 284 & 7751 \\
		& $1,200,000$ & 00:07:03 & 00:00:14 & 22  & 343 & 10951 \\
		\hline
		\multirow{4}*{\tabincell{c}{$1000$}}
		& $100,000$ & 00:00:49 & 00:00:10 & 29  & 213 & 3161 \\
		& $200,000$ & 00:01:37 & 00:00:12 & 31  & 234 & 4471 \\
		& $600,000$ & 00:04:19 & 00:00:17 & 22  & 298 & 7751 \\
		& $1,200,000$ & 00:08:18 & 00:00:23 & 21  & 348 & 10951\\
		\hline
		\multirow{4}*{\tabincell{c}{$2000$}}
		& $100,000$ &  00:01:35 & 00:00:26 & 25  & 196 & 3161 \\
		& $200,000$ & 00:02:54 & 00:00:29 & 25  & 241 & 4471 \\
		& $600,000$ & 00:08:04 & 00:00:35 & 23  & 302 & 7751 \\
		& $1,200,000$ & 00:16:02 & 00:00:50 & 22  & 356 & 10951 \\
		\hline
	\end{tabular}
\end{table}

The numerical performance of the AS strategy applied to the generation of solution paths for the Lasso linear regression model on synthetic data sets with sparsity level at $1\%$ is shown in Table \ref{tab: AS-PPDNA-lasso-Tab3}, and that for the exclusive lasso linear regression model is shown in Table \ref{tab: AS-PPDNA-exclusivelasso-Tab3}. As demonstrated, the AS strategy once again showcases its superior performance for solving sparse optimization models by greatly accelerating the path generation and significantly reducing the problem dimensions.

}

\end{appendix}

\bibliography{references.bib}{}
\bibliographystyle{siam}

\end{document}